\def\arxiv{1}
\documentclass[onefignum,onetabnum]{siamart220329}
\usepackage[sort&compress,numbers]{natbib}
\usepackage{mathtools}
\usepackage{amsmath,amssymb}
\usepackage[normalem]{ulem}
\usepackage{footnotebackref}
\ifdefined\arxiv
\usepackage[margin=.8in]{geometry}
\else
\fi

\usepackage{xcolor}
\usepackage[ruled,linesnumbered,noend,algo2e]{algorithm2e}

\newcommand{\dist}{\mbox{\rm dist}}
\newcommand{\prox}{\mbox{\rm prox}}

\newcommand{\g}{\nabla f(x)}

\newcommand{\norm}[1]{{\left\|{#1}\right\|}}
\newcommand{\R}{\mathbb{R}}

\newcommand{\I}{\text{Id}}
\newcommand{\bxt}{\bar x_t}

\renewcommand{\H}{\mathcal{H}}
\newcommand{\inprod}[2]{\langle#1,\,#2\rangle}
\DeclareMathOperator*{\argmin}{\arg\,min}

\def\citet{\cite}
\def\citep{\cite}
\newcommand{\cS}{{\cal S}}
\newcommand{\cK}{{\cal K}}

\newsiamremark{remark}{Remark}

\newcommand{\TheTitle}{Revisiting Superlinear Convergence of Proximal
Newton-Like Methods to Degenerate Solutions}

\title{{\TheTitle}\thanks{Version of \today.}}

\date{}
\author{Ching-pei Lee\footnotemark[1]\thanks{Institute of Statistical
	Mathematics, and Graduate University for Advanced Studies,
	SOKENDAI.
	Email:
	\url{leechingpei@gmail.com}}
\and
Stephen~J. Wright\thanks{University of Wisconsin-Madison. Email:
	\url{swright@cs.wisc.edu}}}
\begin{document}

\maketitle
\begin{abstract}
We describe inexact proximal Newton-like methods for solving
degenerate regularized optimization problems and for the broader problem of
 finding a zero of a generalized equation that is the sum of a
continuous map and a maximal monotone operator.
Superlinear convergence for both the distance to the solution
set and a certain measure of first-order optimality can be achieved
under a H\"olderian error
bound condition, including for problems in which the continuous map is
nonmonotone, with Jacobian singular at the solution and not Lipschitz.
\ifdefined\arxiv
Superlinear convergence is attainable even when
the Jacobian is merely uniformly continuous, relaxing the standard
Lipschitz assumption to its theoretical limit.
\fi
For convex regularized optimization problems,
we introduce a novel globalization strategy that ensures strict
	objective decrease and avoids the Maratos effect, attaining local
	$Q$-superlinear convergence 
    without prior knowledge of problem parameters.
Unit step size acceptance in our line search strategy does not rely on
continuity or even existence of the Hessian of the smooth term in the objective,  making
the framework compatible with other potential candidates for
superlinearly convergent updates.
\end{abstract}


\section{Introduction}

Consider the following generalized equation:
\begin{equation}
	\text{Find} \, x \in \H \quad \text{such that} \quad 0 \in (A+B)(x),
	\label{eq:fmono}
\end{equation}
where $\H$ is a real Hilbert space with an inner product
$\inprod{\cdot}{\cdot}$ and induced norm $\norm{\cdot}$, 
$A:\H \to \H$ is continuously differentiable,
$B: \H
\rightrightarrows 2^{\H}$ is a maximal-monotone set-valued operator,
the solution set
$\cS$ is nonempty (and closed), and $A$ is locally $L$-Lipschitz continuous in a
neighborhood of $\cS$.
A point $x$ solves \cref{eq:fmono} if and only if the forward-backward step $R(x)$ defined by
\begin{equation}
	R(x) \coloneqq x - \left( \I + B \right)^{-1}\left( \I - A \right)\left( x \right)
\label{eq:rmono}
\end{equation}
is zero, where $\I$ is the identity operator for $\H$.
We define the norm of the forward-backward step as
\begin{equation}
r(x) \coloneqq \norm{R(x)}.
\label{eq:rdefmono}
\end{equation}
Denoting the solution set of \cref{eq:fmono} by $\cS \coloneqq \{ x \in \H \, | \, R(x) =0 \}$, we further assume that the following order-$q$ H\"olderian error bound condition 
holds locally to $\cS$:
\begin{equation}
	\dist (x, \cS) = \dist\big( x, ( A + B)^{-1} (0) \big) \leq
	\kappa r(x)^q, \quad \forall x \in \left\{ x \mid r(x) \leq \epsilon \right\},
\label{eq:qmetricmono}
\end{equation}
for some $\kappa$, $\epsilon > 0$ and $q \in (0,1]$, where $\dist(x,\cS)$ is the distance between the point $x$ and the set $\cS$, measured by the endowed norm on $\H$.
Without assuming knowledge of the values of $\kappa, \epsilon$, or $q$, we propose algorithms that attain fast local rates for a certain range of values of the exponent $q$.
A well-known sufficient condition for \cref{eq:qmetricmono} is the H\"olderian metric subregularity condition of the same order $q$.

Some of our results assume also that the Jacobian $\nabla A$ is
$p$-H\"older continuous in a neighborhood $U$ of $\cS$ for some
\ifdefined\arxiv
$p \in
(0,1]$, while others only assume uniform continuity.
\else
$p \in
(0,1]$.
\fi
We say that $\nabla A$ is $p$-H\"older continuous in $U$ if, for some $\zeta \geq 0$, we have
\begin{equation} \label{eq:DAholder}
\norm{\nabla A(x) - \nabla A(y)} \leq \zeta \norm{x - y}^p,\quad
	\forall x,y\in U \supset \cS.
\end{equation}
For simplicity, we often assume that $U$ is the same as the neighborhood in which \cref{eq:qmetricmono} holds, that is, $U = \{x \mid r(x) \leq \epsilon \}$.
Through simple calculus, we then have
\begin{align}
\nonumber
	\norm{A(x) - A(y) - \nabla A (y) (x - y)} & \leq \frac{\zeta}{1+p}
	\norm{x - y}^{1+p} \\
 & = O(\norm{x-y}^{1+p}),\quad \forall x,y\in U.
\label{eq:smoothpmono}
\end{align}
$\nabla A$ is Lipschitz continuous when it is H\"older continuous with $p=1$.

In this work, we analyze a damped variant of a forward-backward method with Newton-like scaling for \cref{eq:fmono} under conditions of possible degeneracy. 
We account for cases in which the Jacobian of $A$ at any $x^* \in \cS$ could be singular and the solution set $\cS$ may not be compact.
We do not assume that the iterates $\{ x_t \}$ have a limit or an accumulation point.
We will first analyze local convergence under such degeneracy conditions, assuming that the H\"olderian error bound \cref{eq:qmetricmono} and the H\"older continuity condition on the Jacobian \cref{eq:DAholder} are satisfied in a neighborhood of the solution set $\cS$.
For a certain range of values of $p$ and $q$, we prove that both
$r(x_t)$ and $\dist(x_t,\cS)$ converge superlinearly to $0$ and that
the full sequence of iterates $\{ x_t \}$ converges.
\ifdefined\arxiv
We show further
\else
It can further be shown
\fi
that for the most widely considered case of	$q=1$ in \cref{eq:qmetricmono},
	superlinear convergence is still attainable even if the assumption of H\"older continuity of $\nabla A$ is relaxed to local uniform continuity, provided that some algorithm parameters are adapted.

We then discuss the special case of regularized optimization
\begin{equation}
	\min_{x \in \H}\, F(x) \coloneqq f(x) + \Psi(x),
	\label{eq:f}
\end{equation}
where $\Psi: \H \rightarrow [-\infty,\infty]$ is convex, proper, and
closed; and $f: \H \to \R$ is twice continuously differentiable with Lipschitz continuous gradient in an open set containing the domain of $\Psi$.
We can express \cref{eq:f} in the form \cref{eq:fmono} by setting $A(x) \coloneqq \nabla f(x)$ and $B(x) \coloneqq \partial \Psi(x)$.
When specialized to \eqref{eq:f}, the algorithmic framework considered in this work (see \eqref{eq:local} below) is often called {\em proximal-Newton} or {\em sequential quadratic approximation}.
We describe a globalization strategy for \cref{eq:f} that
ensures global convergence and strict decrease of the objective even when conditions \cref{eq:qmetricmono} and \cref{eq:DAholder} do not hold.
When \cref{eq:qmetricmono} {\em is} satisfied and $f$ is convex, the strategy
guarantees eventual acceptance of the unit step size without
requiring knowledge of problem-dependent parameters.
This leads to fast local superlinear convergence for
not only $r(x_t)$ and $\dist(x_t,\cS)$ but also the objective value
$F(x_t)$, provided $\nabla^2 f$ is H\"older
\ifdefined\arxiv
(or uniformly)
\fi
continuous.
Notably, in the scenario $q=p=\rho=1$, our strategy achieves
quadratic convergence, unlike existing backtracking
and trust-region-like approaches.

Our algorithm for \eqref{eq:fmono} is an inexact  forward-backward method with Newton-like scaling,
	for which iteration $t$ has the following form:
\begin{equation}
	\label{eq:local}
	x_{t+1} \approx \left(  H_t + B \right)^{-1}
	\left(H_t - A \right)  (x_t), \quad H_t \coloneqq \left( \mu_t \I + J_t
	\right),\quad
\mu_t \coloneqq c r ( x_t )^\rho,
\end{equation}
where $r(x_t)$ is defined in \cref{eq:rdefmono,eq:rmono}, $c > 0$ and $\rho \geq 0$ are parameters, and $J_t$ is a positive semidefinite but possibly non-Hermitian linear operator satisfying
\begin{equation}
	\norm{J_t - \nabla A(x_t)} = O\big( r( x_t)^\theta \big), \quad \mbox{for some $\theta  \geq \rho$}.
\label{eq:Bmono}
\end{equation}
When $A$ is maximal monotone  and differentiable, $\nabla A(x)$ is positive semidefinite for any $x$ \citep[see Lemma~3.5]{XiaLWZ16a}, so we could simply set $J_t= \nabla A(x_t)$ in this case.
But we allow other scenarios too, such as when $A$ is maximal monotone only in some region.
The scheme \cref{eq:local} can be viewed as replacing $L \cdot \I$ in the classical forward-backward splitting scheme (where $L$ is an upper bound on the Lipschitz constant for $A$)
by $H_t = \mu_t \I + J_t$ as a variable-metric variant.
Because the resolvent $(\I+B)^{-1}$ of $B$ is well-defined and single-valued due to the maximal monotonicity of $B$, so is $(H_t + B)^{-1}$, by positive semidefiniteness of $J_t$.

We can also view the scheme \eqref{eq:local}  as a generalization of the Newton method for solving the nonlinear equation
\begin{equation}
	\label{eq:nonlineareq}
	\text{Find} \, x \in \H \quad \text{ such that }\quad A(x) = 0,
\end{equation}
which is a special case of \cref{eq:fmono} with $B \equiv 0$.

We denote the {\em exact} solution for the next-iterate formula \eqref{eq:local} by $\hat{x}_{t+1}$, that is,
\[
	\hat x_{t+1} \coloneqq \left( H_t + B \right)^{-1} \left(H_t - A
	\right)  (x_t),
\]
equivalently,
\[
	\text{find} \;\; \hat x_{t+1} \in \H\quad \text{ such that} \quad 0 \in \left( B + H_t \right) (\hat x_{t+1}) - \left(H_t -
		A  \right) (x_t).
\]
If we define
\begin{equation}
\label{eq:residual}
	\hat R_t(x) \coloneqq
	x - \left( \I+B \right)^{-1} \left(\left(\I - H_t \right) (x) +
	\left(H_t - A \right) (x_t) \right), \quad
 \hat r_t(x) \coloneqq \| \hat R_t(x) \|,
\end{equation}
then we  have $\hat r_t(\hat x_{t+1})=0$.
Using this residual, we consider the following requirement on the inexactness of the next iterate $x_{t+1}$:
\begin{equation}
\label{eq:stopmono}
\hat r_t(x_{t+1})
\leq \nu r(x_t)^{1+\rho},
\end{equation}
for some parameter $\nu \geq 0$, where $\rho$ has the same meaning as in  \cref{eq:local}, where it is used to define $\mu_t$.

\subsection{Related Works}
Most analyses of \cref{eq:fmono} focus on fixed-point algorithms and global convergence properties.
There are fewer works on the variable-metric framework of \cref{eq:local}, despite its practical interest.
Newton and quasi-Newton approaches for the special case  \cref{eq:nonlineareq} (where $B \equiv 0$)  are well-studied for both  local and global properties under various globalization strategies (see, for example, \cite[Chapter~11]{NocW06a}).
To our knowledge, there is no systematic treatment for extensions \eqref{eq:local} of these approaches to the general case \cref{eq:fmono}.

For regularized optimization \cref{eq:f}, there is a substantial literature on algorithms like \cref{eq:local}, sometimes known in this case as proximal-(quasi-)Newton or sequential quadratic approximation.
Following the classical analysis for Newton's method, Lee et
al.~\citet{LeeSS14a} assumed $f$ strongly convex  and  Lipschitz twice
continuously differentiable, and showed that $x_t$ approaches the unique minimizer $x^*$ $Q$-superlinearly and $Q$-quadratically for inexact and exact proximal-Newton methods, respectively, where $\mu_t \equiv 0$ in \cref{eq:local}.
Later, Yue et al. \citet{YueZS19a} showed that for  $f$ convex and
Lipschitz twice  continuously differentiable with a Lipschitz gradient, 
convergence of a variant
of the damped proximal-Newton approach \cref{eq:local} (which they
termed IRPN) is locally $Q$-superlinear for  $\rho \in (0,1)$, and
$Q$-quadratic if $\rho = 1$ and the unit step size is eventually
always accepted in their Armijo line search.
Convergence here is for both  $r(x_t)$ and $\dist(x_t, \cS)$ to $0$ when the Luo-Tseng error-bound condition holds.
(The latter condition corresponds to \cref{eq:qmetricmono} with $q = 1$.)
A disadvantage of the quadratic convergence result in \citet{YueZS19a}, pointed out by Mordukhovich et al.~\citet{MorYZZ20a}, is that the Maratos effect must be avoided by selecting the line search parameters carefully to satisfy certain conditions involving the Lipschitz constants and the coefficient $\kappa$ in \cref{eq:qmetricmono}.
(The Maratos effect occurs when the unit step size might not yield sufficient decrease in the objective, causing unit steps that make good progress toward the solution to be rejected.)
Selection of these parameters can be difficult in practice.

Mordukhovich et al.~\citet{MorYZZ20a} showed that, under the
same assumptions on $f$ as \cite{YueZS19a}, the conditions for superlinear
	convergence can be relaxed to \cref{eq:qmetricmono} with $q >
	1/2$, although a slightly weaker, $R$-superlinear convergence for
	$\dist(x_t,\cS)$ is obtained.
They also established a convergence rate faster than quadratic when $q >
1$, but as we will see in the proof of \cref{lemma:ytmono}, $r(x) =
O(\dist(x,S))$ always holds, so \cref{eq:qmetricmono} with $q > 1$ cannot hold in regimes of interest.\footnote{Technically speaking, they assumed a $q$-metric subregularity condition such that $\dist(x,\cS) \leq \kappa \dist(0, (A+B)(x))^q$ and showed that this condition is equivalent to \cref{eq:qmetricmono} with the same $q$ {\em provided that} $q \in (0,1]$ in their Proposition~2.4.
They then utilized \cref{eq:qmetricmono} to obtain their convergence rates under the assumption of $q$-metric subregularity. For the case of $q > 1$, they leveraged the same proposition and claimed that such an equivalence continues to hold to obtain their rate --- but the proof of their Proposition~2.4 shows that $q$-metric subregularity with $q > 1$  implies \cref{eq:qmetricmono} only with $q = 1$. Therefore, for their proof of the faster-than-quadratic rate to hold, it is necessary to assume that  \cref{eq:qmetricmono} holds with $q > 1$.}%

	To avoid the  Maratos effect in the regime of quadratic convergence, \citet{MorYZZ20a} proposed a hybrid strategy that accepts unit steps when $r(x)$ decreases at a specified linear rate; they do not always check the objective value in \eqref{eq:f}, so it is not guaranteed to be monotonically decreasing.
Doikov and Nesterov~\cite{DoiN21a} studied the global convergence of a damped Newton approach for \eqref{eq:f} that is a special case of \cref{eq:local}.
When both $f$ and $\Psi$ are convex and  $f$ is $H$-Lipschitz twice continuously differentiable, setting $\rho = 0.5$ and $c = \sqrt{H/3}$ in \cref{eq:local} and solving the subproblem exactly (that is, $\hat{r}_t(x_{t+1})=0$ in \eqref{eq:residual})  leads to strictly decreasing objective values and a global convergence rate of $O(t^{-2})$ for the objective value to the optimum.
No special globalization strategies are required for this approach.

More recently, Liu et al.~\citet{LiuPWY24a} extended the algorithm of
\cite{MorYZZ20a} to nonconvex $f$ restricted to the specific form of
$f(x) = h(Ax-b)$ for some matrix $A$, vector $b$ and a
twice-differentiable function $h$ whose Hessian is implicitly assumed to be locally Lipschitz. T
Their approach requires computing the smallest eigenvalue of $\nabla^2 h$.
Following an earlier preprint version of \cite{MorYZZ20a},
they proved $Q$-superlinear convergence for $q > (\sqrt{5}-1)/2 \approx .618$
under the assumption that the iterates have an accumulation point
$x^*$ with $\nabla^2 f(x^*)$ positive semidefinite.  This requirement of local convexity is similar to our assumption that $A$ is maximal monotone locally.
When $q=1$, their algorithm does not provide quadratic convergence as those variants in \cite{YueZS19a,MorYZZ20a}.
Vom Dahl and Kanzow~\cite{DahK24a} proposed a trust-region-like
variant of the method in \cite{LiuPWY24a} and obtained convergence results similar
to those of that paper.

\subsection{Contributions}

We advance the state of the art in several respects.
\begin{enumerate}
\item In the degenerate setting where the Jacobian $\nabla A$ is
	singular at the solutions,We establish that superlinear convergence of
	proximal-Newton methods is significantly more robust to lack of
	smoothness than previously understood.
	Our analysis relaxes the standard assumption of Lipschitz
	continuity of the Jacobian to H\"older continuity of
	any order while retaining superlinear rates.
	\ifdefined\arxiv
	We further show
	\else
	Using our proof techniques, it can further be shown
	\fi
	that superlinear convergence remains attainable even
	when this assumption is relaxed to merely uniform continuity,
	provided that the damping term and the stopping tolerance decay
	sufficiently slowly.
\item We propose a novel general line search strategy for convex
	regularized optimization that ensures strict function decrease at
	every iteration, guarantees global convergence to $\cS$, and
	avoids the Maratos effect for any ``fast direction'' yielding
	superlinear convergence.
	Our strategy is agnostic to problem-dependent parameters
	and maintains acceptance of the unit step size even in difficult
	scenarios where existing backtracking and trust-region
	approaches fail, such as the quadratic convergence regime
	and the non-Lipschitz Hessian setting described above.
	Notably, our analysis of unit step size acceptance does not
	rely on continuity properties or even the existence of the Hessian,
	but instead leverages the interplay between the H\"olderian error
	bound condition and properties of the update direction.
	This feature makes the framework compatible with any direction yielding
	superlinear convergence,
	including potential candidates like proximal semismooth Newton or
	proximal quasi-Newton methods for more difficult scenarios, such as when $f$ is not twice-differentiable.
\item Our line-search framework achieves $Q$-superlinear
	convergence for the objective value $F$ in \eqref{eq:f}.
	We believe this result to be the first of its kind, even for
	nondegenerate instances of \eqref{eq:f} in which $\nabla^2 f$ is
	nonsingular.
\item We extend local convergence rates previously established for
	regularized optimization to the generalized equation
	setting \eqref{eq:fmono} in Hilbert spaces. %
	We thus allow the use of a non-Hermitian Jacobian in \cref{eq:local}.
	Through a refined error analysis, we broaden the range of the error
	bound exponent for $Q$-superlinear convergence of $\dist(x_t,\cS)$
	to $1 \geq q > (\sqrt{33}-1)/8 \approx 0.593$,
	improving upon $q=1$ in
	\cite{YueZS19a} and $1 > q >
	(\sqrt{5}-1)/2 \approx 0.618$ in \cite{LiuPWY24a,DahK24a}.
	We obtain the same range of $q > 1/2$ for $R$-superlinear
	convergence of $\dist(x_t,\cS)$ and $Q$-superlinear convergence
	of $r(x_t)$, matching the range in \cite{MorYZZ20a} but for the
	more general problem \cref{eq:fmono}.
\end{enumerate}

\subsection{Notation}
For bounded nonnegative scalar sequences $\{\sigma_k\}$ and $\{ \tau_k \}$, we say that $\sigma_k = o(\tau_k)$ if for any $\beta>0$ we have $\sigma_k \le \beta \tau_k$ for all $k$ sufficiently large.
We say $\sigma_k = O(\tau_k)$ if this bound holds for {\em some} $\beta>0$ and all $k$ sufficiently large.
We say  $\sigma_k = \Omega(\tau_k)$ if there is some $\beta>0$ such that $\sigma_k \ge \beta \tau_k$ for all $k$ sufficiently large.

\subsection{Organization}
\Cref{sec:local} discusses local convergence rates of \cref{eq:local}, including the ranges of $q$, $p$, and $\rho$ that yield $R$- and $Q$-superlinear convergence, respectively.
In \Cref{sec:global} we propose a new globalization strategy for \cref{eq:f}, analyzing its global convergence as well as local $Q$-superlinear convergence for $r(x_t)$, $\dist(x_t, \cS)$, and $F(x_t)$.
Simplification of our algorithm for the case of smooth optimization is discussed in \Cref{sec:simple}.
\ifdefined\arxiv
The result of superlinear convergence when $\nabla A$ is only uniform
continuity is in \cref{app:unicont}.
\fi

\section{Local Convergence}
\label{sec:local}

We describe local convergence properties of our basic algorithm \cref{eq:local}, \cref{eq:stopmono} in this section. Section~\ref{sec:superlin} shows  convergence of $\{ x_t \}$ to the solution set $\cS$, and of the residual $r(x_t)$ to zero, at superlinear rates, under certain conditions.

\subsection{Superlinear Convergence} \label{sec:superlin}
For the problem \eqref{eq:fmono}, we start by showing that under suitable conditions, the sequence $\{x_{t}\}$ defined by
\cref{eq:local}, \cref{eq:stopmono} exhibits local superlinear convergence to the solution set $\cS$ (and quadratic convergence, under stronger assumptions) and that  $r(x_t)$ converges to $0$ at the same rate.
We define the notation
\begin{equation} \label{eq:def.dr}
    d_t \coloneqq \dist(x_t, \cS), \quad r_t \coloneqq r(x_t),
	\quad p_t \coloneqq x_{t+1} - x_t,
\end{equation} 
and derive estimates of some key quantities in terms of these values.
\cref{lemma:ptmono} is a generalization to the setting of \cref{eq:fmono} with H\"older continuous $\nabla A$ of \cite[Lemma~4]{YueZS19a} and \cite[Lemma~4.1]{MorYZZ20a}, which apply to \cref{eq:f} with Lipchitz continuous $\nabla^2 f$. 
\begin{lemma}
Fix an iterate $x_t$ and consider the update scheme \cref{eq:local}, \cref{eq:stopmono}
for \cref{eq:fmono} with $\nu \geq 0$, $\rho
\geq 0$, $A$ single-valued and continuously
differentiable, $B$ maximal monotone, and $\cS \neq \emptyset$.
Assume that in a neighborhood containing both $x_t$ and $x_{t+1}$, $A$ is $L$-Lipschitz continuous for some $L \geq 0$ and $\nabla A$ is $p$-H\"older continuous for some $p \in (0,1]$.
Then we obtain
\begin{equation}
\norm{p_t} = \norm{x_{t+1} - x_t} \leq  O (d_t) +  O\big( \mu_t^{-1} d_t^{1+p} \big) +  O\big(\mu_t^{-1} r_t^{1 + \rho}\big) + O\big(  r_t^{1 + \rho}\big).
\label{eq:qmono}
\end{equation}
\label{lemma:ptmono}
\end{lemma}
\begin{proof}
Let $\bxt \in P_{\cS}(x_t)$, where $P_{\cS}$ denotes
projection onto the solution set $\cS$ using the endowed norm, so that
$d_t = \norm{x_t - \bxt}$.\footnote{The solution set $\cS$
	might be nonconvex and thus the projection could be non-unique.
	However, according to our assumption that $\cS \neq \emptyset$,
	there must exist at least one such $\bxt \in P_{\cS}(x_t)$. 
    Regardless, $d_t$ is uniquely defined.}
Denoting by $\bar{U}$ the neighborhood 
assumed in the lemma, we have from the local Lipschitz continuity of $A$ that
\begin{equation}
	\norm{\nabla A(x)} \leq L, \quad \forall x \in \bar{U}.
\label{eq:uppermono}
\end{equation}
Defining
\begin{equation}
\xi_t \coloneqq \hat R_t\left(x_{t+1}\right),
\label{eq:rtmono}
\end{equation}
we have  after rearranging \cref{eq:residual} that 
\begin{align}
	\nonumber
	\left(H_t- A \right) (x_t) - \left( H_t - \I \right) (x_{t+1}) & \in
	(\I+B) \left( x_{t+1} - \xi_t \right)\\
	\quad \Rightarrow \quad
	\xi_t - H_t (\xi_t) & \in \left(A - H_t\right) (x_t) + \left(H_t + B\right) \left(
	x_{t+1}  - \xi_t \right),
	\label{eq:ximono}
\end{align}
while the condition \cref{eq:stopmono} implies that
\begin{equation}
\quad \norm{\xi_t} \leq \nu r_t^{1+\rho}.
\label{eq:xi2mono}
\end{equation}
Because $\bxt \in \cS$, we have from the optimality condition of
\cref{eq:fmono} that
\begin{align}
\nonumber
-A (\bxt) & \in B( \bxt) \\
\; \Rightarrow \;
H_t \left( \bxt - x_t \right) + A(x_t) - A(\bxt) & \in \left(A -
H_t\right) (x_t) + \left(H_t +
B\right)\left( \bxt \right).
\label{eq:partialmono}
\end{align}
Since $J_t$ is positive semidefinite and $B$ is monotone,
 we have from the
definition of $H_t$ in \cref{eq:local} that
\[
\inprod{v-u}{H_t(v-u)} + \inprod{v-u}{z-w} \geq \mu_t
\norm{u-v}^2, \quad \mbox{for any $(u,w)$ and $(v,z)$ in $\mbox{graph}(B)$.}
\]
We thus obtain from \cref{eq:ximono} and \cref{eq:partialmono} that
\begin{equation}
\label{eq:lower}
\begin{aligned}
&~\inprod{ H_t \left( \bxt - x_t \right) + A\left(x_t\right) -
A\left(\bxt\right) - \xi_t + H_t (\xi_t)}{\bxt - x_{t+1} + \xi_t} \\
\geq&~ \mu_t \norm{\bxt - x_{t+1} + \xi_t}^2,
\end{aligned}
\end{equation}
which together with \cref{eq:xi2mono,eq:local,eq:Bmono} implies that
\begin{eqnarray}
\nonumber
&& \norm{\bxt - x_{t+1} + \xi_t} \\
\nonumber
&\stackrel{\cref{eq:local},\cref{eq:Bmono}}{\leq} &
\mu_t^{-1} \big(\norm{(\mu_t + r_t^\theta) \left( \bxt - x_t \right) + \nabla A\left( x_t \right) \left(\bxt -
	x_t\right) + A\left(x_t \right) - A\left( \bxt
\right) }  \\
\nonumber
&& \quad\quad\quad\quad + (1 + \norm{H_t}) \norm{\xi_t} \big) \\
&\stackrel{\cref{eq:smoothpmono},\cref{eq:uppermono},\cref{eq:xi2mono}}{\leq}&
\mu_t^{-1} \big(O\left(\mu_t d_t\right)  + O\big( d_t^{1+p} \big)
 + \nu \left( 1 + L + O\left(\mu_t\right)
\right)r_t^{1+\rho}\big),
\label{eq:firsttermmono}
\end{eqnarray}
where in the final inequality, we used the fact that $\theta \ge \rho$
implies $r_t^{\theta} = O(r_t^\rho) = O(\mu_t)$ and set $y = x_t, x =
\bar x_t$ in applying \cref{eq:smoothpmono}.
We also used the fact $\norm{H_t} = O(1)$ derived from the below with $\theta \geq \rho > 0$:
    \begin{equation*}
    \norm{H_t} \stackrel{\cref{eq:local},\cref{eq:Bmono}}{=} \norm{\mu_t I + \nabla A(x_t) + O(r_t^{\theta})} \stackrel{\cref{eq:local}}{\leq} O(r_t^\rho) + \norm{\nabla A(x_t)} \stackrel{\cref{eq:uppermono}}{\leq} O(r_t^{\rho}) + L = O(1).
    \end{equation*}
Finally, using the triangle inequality, we can bound $\norm{x_{t+1} - x_t}$ by
\begin{eqnarray*}
&&\norm{x_{t+1} - x_t} \\
&\leq& \norm{\bxt - x_{t+1} + \xi_t} + \norm{x_t -
\bxt} + \norm{\xi_t}\\
&\stackrel{\eqref{eq:firsttermmono},\cref{eq:xi2mono}}{\leq}& \mu_t^{-1} \left(
O\left( \mu_t d_t  \right)+O\big( d_t^{1+p} \big) +  O\big(1+ \mu_t \big) O\big(r_t^{1+\rho}\big)\right) + d_t + O\big(
r_t^{1+\rho}\big),
\end{eqnarray*}
verifying \cref{eq:qmono}.
\end{proof}

\begin{lemma}
Suppose the assumptions of  \cref{lemma:ptmono} hold.
Then for $x_{t+1}$ satisfying \cref{eq:stopmono}, we have
\begin{equation}
r_{t+1} = O\big(d_t^{1+\rho} \big) +
O\big( d_t^{1+p} \big) + O\big( \big(r_t^{-\rho}
d_t^{1+p}\big)^{1+p} \big).
\label{eq:gtmono}
\end{equation}
\label{lemma:ytmono}
\end{lemma}
\begin{proof}
We have
\begin{equation}
\label{eq:trimono}	
\begin{aligned}
r_{t+1}
& \leq \| R \left( x_{t+1} \right)
- \hat R_t \left( x_{t+1} \right) \| + \| \hat R_t \left( x_{t+1}\right) \| \\
&= \norm{(\I+B)^{-1}(\I-A) (x_{t+1}) - (\I+B)^{-1} 
	\left( \left( H_t - A \right) (x_t) - \left(H_t - \I \right)
	(x_{t+1})\right)}\\
&\qquad + \hat{r}_t(x_{t+1}).
\end{aligned}
\end{equation}
From the nonexpansiveness of the resolvent of $B$, we can bound the
first term on the right-hand side of \cref{eq:trimono} by
\begin{equation} 
\label{eq:prox1mono}
\begin{aligned}
&~\norm{(\I+B)^{-1}(\I-A) (x_{t+1}) - (\I+B)^{-1}
	\left( \left( H_t - A \right) (x_t) - \left(H_t - \I \right)
	(x_{t+1})\right)}\\
	\leq &~\norm{ A (x_t) - A  (x_{t+1}) - H_t \left(  x_t -
	x_{t+1}\right)}.
\end{aligned}
\end{equation}
Next, the fact that $\bar x_{t} \in P_{\cS} \left( x_{t}
\right)$ is a solution of \eqref{eq:fmono} indicates that
\[
	(\I+B)^{-1} (\I - A) (\bar x_{t}) = \bar x_{t}.
\]
From the nonexpansiveness of the resolvent of $B$ and the expressions above, we obtain by applying the triangle inequality again that
\begin{equation}
\label{eq:Lipmono}
\begin{aligned}
	r\left(x_{t}\right) &=
	\norm{x_{t} - (\I+B)^{-1}(\I-A) (x_{t})} \\
	&= \norm{\left(x_{t} -
	(\I+B)^{-1}(\I-A) (x_{t})\right) - \left(\bar x_{t}
	-(\I+B)^{-1}(\I-A) (\bar x_{t})\right)}\\
	&\leq \norm{\bar x_{t} - x_{t}} + \norm{\bar x_{t} - x_{t}
	+ A (x_{t}) - A (\bar x_{t})}\\
	&\leq (L + 2) \norm{\bar x_{t} - x_{t}} = (L+2) d_t,
\end{aligned}
\end{equation}
where we used local Lipschitz continuity of $A$ in the last inequality.
We also note from the definition of $\mu_t$ in \cref{eq:local} that
\begin{equation}
\mu_t^{-1} r_t^{1+\rho} = c^{-1} r_t \stackrel{\cref{eq:Lipmono}}{=}
O(d_t).
\label{eq:gmono}
\end{equation}
By substituting \cref{eq:prox1mono} and \cref{eq:stopmono} into \cref{eq:trimono} and using the definition $p_t \coloneqq x_{t+1} - x_t$, we can prove \cref{eq:gtmono} as follows:
\begin{eqnarray}
\nonumber
&&r\left( x_{t+1} \right)\\
\nonumber
&\leq
&\norm{A (x_t) -  A\left( x_{t+1} \right) - H_t \left( x_t -
x_{t+1}\right)}+ \nu r_t^{1+\rho}\\
\nonumber
&\stackrel{\cref{eq:local},\cref{eq:Bmono}}{=}
&\norm{\nabla A(x_t) \left( x_{t+1} - x_t \right) + A (x_t) -  A\left(
x_{t+1} \right) + (\mu_t + O(r_t^\theta)) \left(x_{t+1} - x_t
\right)}\\
\nonumber
&&\qquad + \nu r_t^{1+\rho}\\
\label{eq:rsuper}
&\stackrel{\cref{eq:smoothpmono},\cref{eq:local},\cref{eq:Bmono},\cref{eq:def.dr}}{\leq}
&O\big( \norm{p_t}^{1+p}\big) + O\left(\mu_t\right) \norm{p_t} + \nu
	r_t^{1+\rho}.
\end{eqnarray}
Thus from \Cref{lemma:ptmono}, we obtain
\begin{eqnarray*}
r_{t+1}  &\leq 
&O\Big( d_t^{1+p} + \big( \mu_t^{-1} d_t^{1+p} \big)^{1+p}
	+ \big( \mu_t^{-1} r_t^{1+\rho} \big)^{1+p} + \big(
	r_t^{1+\rho}\big)^{1+p} \Big) \\
\nonumber
	&& \quad  + O(\mu_t d_t)+ O\big( d_t^{1+p}
	\big) + O\big( r_t^{1+\rho}\big) + O\big( \mu_tr_t^{1+\rho} 
	\big)\\
\nonumber
&\stackrel{\cref{eq:Lipmono},\cref{eq:gmono}}{=}
&O\big( d_t^{1+p} \big) + O\Big(
\big( r_t^{-\rho}d_t^{1+p} \big)^{1+p} \Big)+ O\big( d_t^{1+\rho} \big),
\end{eqnarray*}
where in the last equality we used $\rho>0$ to deduce that $(1 + \rho) (1 + p)
\geq  (1+p)$ and $1 + 2\rho \geq 1+ \rho$.
\end{proof}
It follows from \cref{eq:Lipmono} that $q > 1$ in \cref{eq:qmetricmono} is possible only in trivial circumstances.
Our main superlinear convergence is as follows.
\begin{theorem}
Consider the update scheme \cref{eq:local}, \cref{eq:stopmono}
for \cref{eq:fmono} for some $\nu, \rho
\geq 0$, with $A$ single-valued and continuously
differentiable, $B$ maximal monotone, and $\cS \neq \emptyset$.
Assume that \cref{eq:qmetricmono} holds for some $q > 0$ in a neighborhood $V \coloneqq \{x \mid r(x) \leq \epsilon\}$ of $\cS$ for some $\epsilon > 0$, and that $A$ is $L$-Lipschitz continuous for some $L \geq 0$ and $\nabla A$ is $p$-H\"older continuous \eqref{eq:DAholder} for some $p \in (0,1]$ within the same neighborhood.
Suppose that the following inequalities are satisfied:
\begin{equation}
\left\{ 
\begin{aligned}
(1 + \rho) q &> 1,\\
	(1 + p) q &>1,\\
	\left(q + pq - \rho\right)  (1+p) &> 1.
 \end{aligned}
\right.
\label{eq:ineqmono}
\end{equation}
Then if $r_0$ is sufficiently small, 
we have $Q$-superlinear convergence of $\{r_t\}$ and $\{ d_t \}$ according to
\begin{equation}
r_{t+1} = O\left( r_t^{1+s} \right),\quad
d_{t+1} = O\left( d_t^{1+s} \right),
\label{eq:superconvmono}
\end{equation}
where $s$ is the smallest gap between the left- and right-hand sides in \eqref{eq:ineqmono}, that is, 
\begin{equation*}
s \coloneqq \min\big\{ (1 + \rho) q, \, (1 + p) q, \, (q + pq -
	\rho) ( 1+p)  \big\} - 1 > 0.
\end{equation*}
\label{thm:supermono}
\end{theorem}
\begin{proof}
Suppose for some $t \ge 0$ that $r_t \le \epsilon_1$ for some
$\epsilon_1 \in (0,\epsilon]$ whose value is defined below.
From \cref{eq:qmetricmono} we have
\begin{equation}
	r_t^{-\rho} = \left( r_t^q \right)^{-\frac{\rho}{q}}
	= O\big(d_t^{-\frac{\rho}{q}}\big).
	\label{eq:gdmono}
\end{equation}
By substituting \cref{eq:gdmono} into \cref{eq:gtmono} we then have
\begin{align}
	\nonumber
	r_{t+1}&= O\big(d_t^{1+\rho} \big) + O\big( d_t^{1+p} \big) + O\big(
\big(d_t^{-\frac{\rho}{q}} d_t^{1+p}\big)^{1+p} \big)\\
&= O\Big( d_t^{1+\rho} + d_t^{1+p} + d_t^{\left(1+p -
	\frac{\rho}{q}\right)(1+p)}\Big).
\label{eq:finalformmono}
\end{align}
By applying \cref{eq:qmetricmono} to the right-hand side of \cref{eq:finalformmono}, we obtain the first equation in \cref{eq:superconvmono}, as well as $r_{t+1} = O(\epsilon_1^s r_t)$.
Because $s>0$, we can decrease  $\epsilon_1$ if necessary to ensure that $r_{t+1} < \frac{99}{100} r_t \le \frac{99}{100} \epsilon_1$ for {\em all} $x_t$ with $r_t \le \epsilon_1$, showing that $x_{t+1} \in V$.
We can thus also apply %
\cref{eq:finalformmono} to \cref{eq:qmetricmono} to obtain
\[
d_{t+1} = O\left(r_{t+1}^q\right)
= O\Big( d_t^{(1+p)q} + d_t^{(1+\rho)q} +
d_t^{(q+pq-\rho)(1+p)}\Big)
=O\left( d_t^{1+s}
\right),
\]
proving the case of $Q$-superlinear convergence for $\{d_t\}$.
The proof is completed by noting that if $x_0$ is such that $r_0 \le \epsilon_1$, then $\{ r_t \}$ decreases monotonically to zero.
\end{proof}

If we seek only $R$-superlinear convergence for $\{d_t \}$, while
retaining $Q$-superlinear convergence for $\{r_t\}$, a different range
of parameters is allowed. %
\begin{theorem}
\label{thm:supermono2}
Suppose that the assumptions of \cref{thm:supermono} hold, except that in place of \cref{eq:ineqmono}, the following inequalities are satisfied:
\begin{equation}
\left\{
\begin{aligned}
	(1 + p) q &> 1,\\
	\left(q + pq - \rho\right)  (1+p) &> 1,\\
	\rho + q &> 1,\\
	\rho &>0.
\end{aligned}
\right.
\label{eq:ineqmono2}
\end{equation}
Then we obtain $Q$-superlinear convergence within $V$ for $\{r_t\}$ according to
\begin{equation}
r_{t+1} = O\left( r_t^{1+\bar{s}} \right),\quad
\label{eq:superconvmono2}
\end{equation}
where $\bar{s}$ is the smallest gap between the left- and right-hand sides in \eqref{eq:ineqmono2}, that is,
\begin{equation}
\bar{s} \coloneqq \min\big\{(1 + p) q, \, (q + pq -
	\rho ) ( 1+p), \, \rho + q, \, 1 + \rho \big\} - 1 > 0.
	\label{eq:s2}
\end{equation}
Moreover, we have $R$-superlinear convergence within $V$ for $\{d_t\}$.
\end{theorem}
\begin{proof}
As in the proof of \cref{thm:supermono}, suppose that $r_t \le
\epsilon_1 \le \epsilon$ for some $\epsilon_1>0$.
From \cref{eq:qmetricmono} we have $d_t = O(r_t^q)$, so that $r_t^{-1} = O(d_t^{-1/q})$. Using this bound together with \Cref{lemma:ptmono} and \cref{eq:Lipmono}, we obtain
\begin{equation}
p_t \coloneqq x_{t+1} - x_t = O(d_t) + O\Big(d_t^{1+p-\frac{\rho}{q}}\Big) =
O\Big(d_t^{\min\left\{ 1, 1+p-\frac{\rho}{q} \right\}}\Big).
\label{eq:pbmono}
\end{equation}
Substitution of \cref{eq:pbmono} into \eqref{eq:rsuper} 
then leads to
\begin{eqnarray*}
r_{t+1} &\leq& O(d_t^{1+p}) + O\Big(d_t^{(1+p -
	\frac{\rho}{q})(1+p)}\Big) + \mu_t O(d_t) + \mu_t O\Big(d_t^{1+p
	- \frac{\rho}{q}}\Big) + \nu r_t^{1+\rho}\\
&\stackrel{\cref{eq:qmetricmono}}{=}&
O(r_t^{q(1+p)}) + O\Big(r_t^{(q+pq - \rho)(1+p)}\Big) +
O(r_t^{\rho+q})  + O\Big(r_t^{\rho + (q+pq-\rho)}\Big)+
O(r_t^{1+\rho})\\
&\stackrel{\cref{eq:s2}}{=}& O(r_t^{1+\bar{s}}),
\end{eqnarray*}
which is exactly \cref{eq:superconvmono2}, and $\bar{s} > 0$ if and only if \cref{eq:ineqmono2} holds.
Note that by choosing $\epsilon_1$ sufficiently small, we can ensure that $r_{t+1} \le \frac{99}{100} r_t \le \frac{99}{100} \epsilon_1$, so by requiring that $r_0 \le \epsilon_1$, we have that $\{r_t \}$ decreases to zero, and in fact converges superlinearly to zero since $\bar{s}>0$.

To prove $R$-superlinear convergence of $\{ d_t \}$,  we note that by defining $\hat d_t = \kappa r_t^q$ and  combining \cref{eq:superconvmono2} and \cref{eq:qmetricmono}, we have
\[
d_t \leq \hat d_t, \quad \hat d_{t+1} = O(\hat d_t^{1+ \bar{s}}), \quad \forall t \geq 0,
\]
thus showing that $\{d_t \}$ is dominated by a $Q$-superlinearly convergent sequence.
\end{proof}

When the conditions of either \cref{thm:supermono} or \cref{thm:supermono2} hold, we
also obtain strong convergence for the iterates to a solution point.
\begin{theorem}
\label{thm:iterate}
Assume that the conditions of either \cref{thm:supermono} or \cref{thm:supermono2} hold.
Then  $\{x_t\}$ converges strongly to some point $x^*\in \cS$.
\end{theorem}
\begin{proof}
\cref{thm:supermono} and \cref{thm:supermono2}  both show that $x_t \in V$ for all $t$ and that $\{r_t\}$ converges superlinearly to zero.
By using \cref{eq:qmetricmono} in \cref{eq:qmono}, we see that $p_t$ is
bounded by $r_t$ as follows:
\[
	\norm{p_t} \leq O(r_t^q) + O \big(r_t^{(1+p)q - \rho} \big) + O(r_t) +
	O \big(r_t^{1+\rho} \big) = O \big(r_t^{\min\{q,1,1+\rho,(1+p)q - \rho} \big).
\]
From the constraints in \cref{eq:ineqmono,eq:ineqmono2}, together with $\rho
\geq 0$ and $q > 0$, we obtain
\[
	\min \big\{1+\rho,(1+p)q - \rho,q,1 \big\} > 0,
\]
so \cref{eq:pbmono,eq:qmetricmono} indicate that
there is $\tau > 0$ such that $\norm{p_t} = O(r_t^\tau)$.
 \cref{thm:supermono,thm:supermono2} show that $\{r_t\}$ converges superlinearly to $0$.
Therefore, we can find $t_1 \ge 0$ and $\eta \in [0,1)$ such that
\begin{equation*}
	r_{t+1} \leq \eta r_t, \quad \forall t \geq t_1.
\end{equation*}
Thus we get
\[
	\sum_{t=t_1}^{\infty} \norm{p_t} = O\bigg( \sum_{t = t_1}^{\infty}
	\eta^{\tau(t-t_1)} r_{t_1}^{\tau} \bigg)
	= O\bigg(\frac{r_{t_1}^{\tau}}{ 1 - \eta^{\tau}}  \bigg) < \infty,
\]
so $\{\norm{p_t}\}$ is summable.
Thus $\{x_t\}$ is a Cauchy sequence, so it converges strongly to a point $x^*$ because of completeness of the  Hilbert space ${\cal H}$.
Moreover, $\dist(x_t,\cS) \rightarrow 0$ implies that  $\dist(x^*,\cS) = 0$,  so by closedness of $\cS$ we have $x^* \in \cS$.
\end{proof}

Comparing \cref{eq:ineqmono} and \cref{eq:ineqmono2},
we see that the difference is that the inequality 
\begin{equation}
(1+\rho)q > 1 
\label{eq:eq1}
\end{equation}
in the former is replaced by the two inequalities
\begin{equation}
\rho+q>1, \quad \rho>0
	\label{eq:eq2}
\end{equation}
in the latter.
Since $q \leq 1$, we have $\rho \geq \rho q$, so
\cref{eq:eq1} implies \cref{eq:eq2}.
We discuss several interesting realizations of \cref{eq:ineqmono,eq:ineqmono2} below.
(To our knowledge, the analysis for generalized equation \eqref{eq:fmono} is new, and for the special case  \eqref{eq:f}, our bounds are new.)
\begin{remark}[Convergence Regimes]
\label{ex:ratesmono}
\begin{enumerate}
\item Quadratic convergence occurs when $p = \rho = q = 1$, that is,
	when the Luo-Tseng error bound holds and $\nabla^2 f$ is Lipschitz
	continuous (for \eqref{eq:f}) or
	when the error bound condition with $q=1$ holds with $\nabla A$ Lipschitz continuous
	(for \cref{eq:fmono}).
	(Our results show that $p=1$ and $q \geq \rho \geq 1$
	imply a convergence rate that is at least quadratic.)
\item If $p=1$ and $q \leq 1$, \cref{eq:ineqmono} implies
\begin{equation*}
\left\{
\begin{aligned}
2q &> 1\\
(1+\rho)q &> 1\\
2 \left(2q - \rho\right) &> 1
\end{aligned}
\right.
\; \Leftrightarrow\;
\left\{
\begin{aligned}
q &> \tfrac{1}{2}\\
\rho &> \tfrac{1-q}{q}\\
2q -\tfrac{1}{2} &> \rho
\end{aligned}
\right.
\; \Leftrightarrow \;
\left\{
\begin{aligned}
q &> \tfrac{1}{2}\\
q &> \tfrac{-1 + \sqrt{33}}{8} \\
\rho &\in (\tfrac{1-q}{q},2q-\tfrac12).
\end{aligned}
\right..
\end{equation*}
(We obtain the second formula in the last column by ensuring
the $\rho$-interval is nonempty.) 
We thus attain $Q$-superlinear convergence of $\{d_t\}$ provided that
$q > \tfrac18 (\sqrt{33}-1) ) \approx 0.593$, as long as  $\rho < 2q-\tfrac12$.
(We can set $\rho = \tfrac14 (\sqrt{33}-1) -\tfrac12  = \tfrac14 ( \sqrt{33}-3 )$ to maximize the allowable range of $q$.)
By comparison with the result of
\citet{MorYZZ20a,LiuPWY24a,DahK24a}, which required $\rho \leq q$, our analysis shows that superlinear convergence can be obtained even in some situations where $\rho > q$.
In addition, \citet{MorYZZ20a} does not guarantee $Q$-superlinear
convergence of $\{d_t\}$, while the result of Yue et al.
\cite{YueZS19a} has such a guarantee only for $q=1$ and
\cite{LiuPWY24a,DahK24a} for $q > \tfrac12 (\sqrt{5}-1) \approx 0.618$.
By contrast, our bound allows a wider range for $q$.
When our analysis is applied to smooth optimization problems, the allowed range is also wider than that in \cite[Theorem~5]{Lee20a}.
For \cref{eq:ineqmono2}, $p=1$ and $q \leq 1$ then imply that
\[
\left\{
	\begin{aligned}
		\rho &> 0,\\
		q &> \tfrac12,\\
		\rho + q &> 1,\\
		4q - 2 \rho & > 1,
	\end{aligned}
 \right.
\]
	so $q > 1/2$ with $\rho = 1/2$ implies $Q$-superlinear convergence
	for $\{r_t\}$ and $R$-superlinear convergence for $\{d_t\}$.

\item If $q = 1$ and $p>0$, then any $\rho \in (0,p)$ implies superlinear convergence.
In contrast, if $q=1$, $p=\rho=0$,  the damping becomes a positive constant, making $J_t$ uniformly bounded. 
Linear convergence could then by obtained from standard analysis of first-order-like variable metric approaches.
\end{enumerate}
\end{remark}

\ifdefined\arxiv
Inspired by the last item in \cref{ex:ratesmono}, we further
demonstrate that the $p$-H\"older assumption on $\nabla A$ can be
relaxed to merely \emph{uniform continuity}, provided the damping and
stopping tolerance decay sufficiently slowly.
We present this result in \cref{app:unicont} as its proof is technical
but largely mirrors the analysis in this section.
\fi

\section{A Line Search Strategy for Convex Regularized Optimization}
\label{sec:global}

In this section, we focus on \cref{eq:f}, the special case of \cref{eq:fmono} obtained by setting $A = \nabla f$ and  $B = \partial \Psi$.
Although solutions of \cref{eq:fmono} in general only correspond to stationary points of \cref{eq:f}, we make the further assumption that $f$ is convex (similar to \cite{YueZS19a,MorYZZ20a}), which causes the solution sets of \cref{eq:fmono,eq:f} to coincide in this case.
In this scenario, $A$ is also maximal monotone, so \cite[Theorem~3.5]{DruL16a} indicates that
\begin{equation}
r(x) \leq \dist(0,(A+B)(x)),    
\label{eq:rtosubgrad}
\end{equation}
or equivalently $r(x) \leq \dist(0, \partial F(x))$ for \cref{eq:f}, for all $x$.
Therefore, the error bound \cref{eq:qmetricmono} further implies a more intuitive local upper bound of $\dist(x,S)$ by $\partial F$.
Moreover, \cite[Theorem~3.4]{MorO15a} shows that
	this bound, together with convexity, implies a growth condition that
	effectively upper bounds $\dist(x,S)$ by the objective gap.
	This characterization serves as a critical tool in our subsequent analysis.

We describe algorithms that attain global convergence to the solution set, and examine their rates of convergence.
To simplify the description of our globalization strategy, we further assume that $\nabla f$ is globally $L$-Lipschitz continuous in its domain, not just locally Lipschitz, as in our discussion above. 

We first discuss an existing algorithm due to \cite{MorYZZ20a},  and argue that it guarantees global convergence and ensures local superlinear convergence of both $r_t$ and $d_t$ to $0$ when the conditions in \cref{thm:supermono} or \cref{thm:supermono2} hold.
We then propose a novel strategy that retains these properties and adds another property, namely, strict  decrease of the objective function at  each iteration.
Importantly, our new approach exhibits $Q$-superlinear convergence of the objective value to its optimal value, which we believe to be a new result even for nondegenerate problems. 
Existing analyses for ensuring local superlinear convergence in proximal-Newton-type methods require Lipschitz continuity of the Hessian of $f$ and depend on a Taylor expansion to guarantee acceptance of the unit step size. 
We use instead a novel mechanism and analysis that relies on the
convexity of $F$ and the H\"olderian error bound condition to
guarantee sufficient function decrease for a unit step.
Note that we do not need to assume that the Hessian $A(x) = \nabla^2 f(x)$ is Lipschitzian.

\subsection{Two Algorithms} \label{sec:lsalg}

In our following discussion of both the existing approach and our new globalization approaches for \cref{eq:f}, a tentative iterate $\tilde x_{t+1}$ is first obtained from the approximate minimization
\begin{equation}
\tilde x_{t+1} \approx \argmin_x \, q_t(x),
\label{eq:newton}
\end{equation}
with
\begin{equation} \label{eq:subprob}
\begin{aligned}
q_t(x)&\coloneqq \inprod{g_t}{x - x_t} + \frac{1}{2} \inprod{H_t
\left( x - x_t \right)}{x - x_t} + \Psi\left( x \right),\\
g_t &\coloneqq \nabla f(x_t), \;\; H_t \coloneqq
J_t + \mu_t \I, \;\; \mu_t \coloneqq c
r(x_t)^{\rho},
\end{aligned}
\end{equation}
where $r(\cdot)$ follows the definition in \cref{eq:rmono,eq:rdefmono}, which can be  written equivalently in this setting as follows:
\begin{equation} \label{eq:defR3}
R(x) = x - \prox_\Psi \left( x - \nabla f(x) \right), \quad
r(x) = \norm{R(x)}, 
\end{equation}
and $J_t$ is a linear operator satisfying\footnote{For
	$J_t$ to satisfy this definition, $\nabla^2 f$ needs to be
	positive semidefinite at any accumulation point $\bar x$ of the
	sequence $\{x_t\}$ that satisfies the stationarity condition $r(\bar{x})=0$. We refer to this property as ``local convexity." Our algorithm can be extended to nonconvex problems by considering other choices of upper-bounded and positive-definite $J_t$, with the  global convergence results of the next subsection continuing to hold. However, this local convexity assumption is necessary for superlinear convergence.}
\begin{equation}
J_t \;\; \mbox{\rm symmetric}, \quad J_t \succeq 0, \quad
\norm{J_t - \nabla^2 f(x_t)} =
O\big( r( x_t )^\theta \big) \text{ for some } \theta \geq \rho.
\label{eq:B}
\end{equation}
A line search procedure is applied to the update direction 
\begin{equation}
	\tilde p_t
	\coloneqq \tilde x_{t+1} - x_t
	\label{eq:tildep}
\end{equation}
to find a suitable step size $\alpha_t > 0$; we then set $x_{t+1} \gets x_t + \alpha_t \tilde p_t$.
To ensure a descent direction and global convergence, the point $\tilde x_{t+1}$ from \eqref{eq:newton} is required to satisfy the conditions
\begin{equation}
\label{eq:stop}
q_t(\tilde x_{t+1}) \leq q_t(x_t),\quad
\hat r_t(\tilde x_{t+1}) \leq \nu \min\left\{ r(x_t)^{1 + \rho}, r(x_t)
\right\},\quad \nu \in [0,1).
\end{equation}
In the setting of regularized optimization, we have that
$\hat r_t(\cdot)$ in \cref{eq:residual} is equivalent to
\begin{equation}
\label{eq:rhat}
	\hat r_t(\tilde x_{t+1}) = \| \hat R_t(\tilde x_{t+1}) \|, \;
	\hat R_t\left( \tilde x_{t+1} \right) = \tilde x_{t+1} -
	\prox_{\Psi}\left(\tilde x_{t+1} - g_t - H_t
\left( \tilde x_{t+1} - x_t \right) \right).
\end{equation}

The conditions above are just specific realizations of \cref{eq:local,eq:Bmono,eq:residual,eq:stopmono} for the case of \cref{eq:f}, except that \cref{eq:stop} contains additional requirements on $q_t$, $\nu$, and $r_t = r(x_t)$.
Thus, the results that we derived for $p_t$ in Section~\ref{sec:local}, such as \cref{eq:pbmono}, apply to $\tilde p_t$ as well.
Moreover, we require  $J_t$ to be symmetric, which is natural in this setting, since it is an approximation of the Hessian $\nabla^2 f(x_t)$.

The first approach we consider, shown in \cref{alg:gdnm}, is due to \cite{MorYZZ20a}. 
(Both this approach and its successor, \cref{alg:newton2}, require the conditions \cref{eq:stop} to hold.)

\begin{algorithm}[tbh]
\DontPrintSemicolon
\SetKwInOut{Input}{input}\SetKwInOut{Output}{output}
\caption{Proximal Newton Method in \cite{MorYZZ20a}}
\label{alg:gdnm}
\Input{$x_0\in\H$, $\beta,\gamma,\sigma\in (0,1)$, $\rho > 0$, $\nu
	\geq 0$, $c > 0$, $\bar C>F(x_0)$}

$\eta \leftarrow r(x_0)$

\For{$t=0,1,\ldots$}{
	Select $H_t$ satisfying \cref{eq:subprob,eq:B} with $\theta = 1$
	and find an approximate solution $\tilde x_{t+1}$ of \cref{eq:newton}
	satisfying \cref{eq:stop}

 $\tilde p_t \leftarrow \tilde x_{t+1} - x_t$

	\If{$t > 0$, $r(\tilde{x}_{t+1})\le\sigma \eta$, 
 and $F(\tilde{x}_{t+1})\le \bar C$}
	{
		$\alpha_t\leftarrow 1$,
  $\eta \leftarrow r(\tilde{x}_{t+1})$
	}
	\Else
	{		
		$\alpha_t \leftarrow \beta^{m_t}$, where $m_t$ is the smallest nonnegative integer $m$ such that
\begin{equation}\label{eq:armijo}
F\left( x_t + \beta^m \tilde p_t \right)\le F\left( x_t
\right)-\gamma \mu_t\beta^m \norm{\tilde p_t}^2.
\end{equation}
}	
Set $x_{t+1}\leftarrow x_t + \alpha_t \tilde p_t$. 
}
\end{algorithm}

If the conditions of \cref{thm:supermono,thm:supermono2} hold, then when $r(x_t)$ is small enough, all subsequent iterations of \cref{alg:gdnm} will accept the unit step through the condition in  Line 5 (no backtracking required), as $Q$-superlinearly convergence of $\{r(x_t)\}$ implies that this sequence also converges $Q$-linearly with an arbitrarily fast rate, and global convergence of $r(x_t)$ to $0$ is guaranteed by \cite[Theorem~3.1]{MorYZZ20a}.
In the latter reference, local superlinear convergence requires local Lipschitz continuity of $\nabla^2 f$, so our analysis in \Cref{sec:local} (which required H\"older continuity of $\nabla^2 f$) slightly broadens the problem class for which \cref{alg:gdnm} is superlinearly convergent.
\cref{thm:supermono} provides additional guarantees for $Q$-superlinear convergence of $\{d_t\}$ not covered by
\cite{MorYZZ20a}. %
Moreover, our requirement for $\theta \geq \rho$ in \eqref{eq:B} is less restrictive than the choice $\theta = 1$ in \cite{MorYZZ20a}.

\cref{alg:gdnm} does not require monotonic decrease of function values, as the conditions of Line 5 can be satisfied without having $F(\tilde{x}_{t+1})<F(x_t)$.
Our new approach, \cref{alg:newton2}, will ensure strict function decrease at each step while retaining $Q$-superlinear convergence of the function values.

The first distinctive element in \cref{alg:newton2} is to replace the sufficient decrease condition \cref{eq:armijo} with 
\begin{equation}
\label{eq:newarmijo}
F(x_{t+1}) \leq F\left( x_t \right)-\gamma \alpha_t^2 \norm{\tilde p_t}^{2+\delta}
\end{equation}
(where $\gamma$, $\tilde p_t$, and $\alpha_t$ are defined as in \cref{alg:gdnm}, but $x_{t+1}$ is defined differently; see below)
for some given $\delta \ge 0$.
Our analysis in the next subsection will show that when the conditions \cref{eq:ineqmono} are satisfied by $p$, $q$, and $\rho$, we can set $\delta=2$.
Since $\tilde p_t \neq 0$ for all $t$, 
the objective value is always decreasing.
The second distinctive element of \cref{alg:newton2}, inspired by a technical result from \cite[Lemma~2.3]{BecT09a} on how $r(x)$ and $\dist(x,S)$ bound the objective gap, is to test the condition \cref{eq:newarmijo} on a point obtained from a proximal gradient step.
In this vein, we first compute
\begin{equation}
    \label{eq:y}
y_{t+1} (\alpha_t) \leftarrow x_t + \alpha_t \left(\tilde x_{t+1}  - x_t 
		\right) = x_t + \alpha_t \tilde p_t,
\end{equation}
and then compute a proximal gradient step from $y_{t+1}(\alpha_t)$ to obtain the candidate $\bar{x}_{t+1}(\alpha_t)$ for the next iterate, as follows
\begin{equation}\label{eq:xbar}
\begin{aligned}
\bar x_{t+1}(\alpha_t) \gets &~ \prox_{(\Psi/L)}\big(
y_{t+1}(\alpha_t) - \tfrac{1}{L}\nabla f(y_{t+1}(\alpha_t))
\big) \\
= &~ y_{t+1}(\alpha_t) - \tfrac{1}{L} G_L(y_{t+1}(\alpha_t)),
\end{aligned}
\end{equation}
where $L$ is the Lipschitz constant for $\nabla f$,
 \begin{equation} \label{eq:defGL}
	G_{L}(x) \coloneqq L \Big(x - \prox_{\frac{\Psi}{L}} \Big(x -
	\frac{1}{L} \nabla f(x) \Big)\Big)
\end{equation}
is the proximal gradient, and $\prox_{g}$ denotes the proximal operator
\begin{equation*}
	\prox_g(x) \coloneqq \min_{y\in \H}\; \tfrac12\norm{y-x}^2 + g(y)
\end{equation*}
for any function $g$.\footnote{For simplicity, we use
        a given (possibly rough) upper bound $L$ on the actual
		Lipschitz constant here, but we can also conduct another
		backtracking on $L$ to remove dependency on knowledge of this
		problem parameter.
        Our analysis still holds (with some additional calculations and notations) as long as $L$ satisfies the following condition:
        \[
        F(\bar x_{t+1}(\alpha_t))
        \leq F(y_{t+1}(\alpha_t)) + \inprod{\nabla f(y_{t+1}(\alpha_t))}{z_{t+1}(\alpha_t)} + \frac{L}{2} \norm{z_{t+1}(\alpha_t)}^2 + \Psi(\bar x_{t+1}(\alpha_t)),
        \]
        where $z_{t+1}(\alpha_t) \coloneqq \bar x_{t+1}(\alpha_t) - y_{t+1}(\alpha_t)$.
    }

The full algorithm can be specified as follows.

\begin{algorithm}[tbh]
\setcounter{AlgoLine}{0}
\DontPrintSemicolon
\SetKwInOut{Input}{input}\SetKwInOut{Output}{output}
\caption{A Proximal Newton Method Guaranteeing Strict Decrease and
Superlinear Convergence for the Objective Value}
\label{alg:newton2}
\Input{$x_0\in\H$, $\beta \in (0,1)$, $\gamma\in (0,1)$, $\nu \in [0,1)$, $\rho \in
(0,1]$,
	$c> 0$, $\delta \ge 0$, Lipschitz constant $L$ for $\nabla f$}

\For{$t=0,1,\ldots$}{
Select $H_t$ satisfying \cref{eq:subprob,eq:B} with $\theta \geq \rho$ and find an approximate solution $\tilde x_{t+1}$
of \cref{eq:newton} satisfying \cref{eq:stop}

	$\alpha_t \leftarrow 1, \quad \tilde p_t \gets \tilde x_{t+1} - x_t$

 Terminate if $\tilde p_t = 0$

	\While{True}
	{
		Compute $y_{t+1}(\alpha_t)$ from  \cref{eq:y} and $\bar x_{t+1}(\alpha_t)$ from \cref{eq:xbar}

		\If{$F(\bar x_{t+1}(\alpha_t)) \leq F(x_t) - \gamma \alpha_t^2 \norm{\tilde p_t}^{2+\delta}$}
		{
			$x_{t+1} \leftarrow \bar{x}_{t+1}(\alpha_t)$ 

			Break
		}
		\lElse{
			$\alpha_t \leftarrow \beta \alpha_t$
		}
	}
}
\end{algorithm}

The step taken at each iteration of \cref{alg:newton2} is a composition of a prox-Newton step (with backtracking) and a short prox-gradient step (with the conservative choice $1/L$ for the steplength parameter).
(One could instead use $\min\{\norm{\tilde p_t}^2, \norm{\tilde p_t}^{2+\delta}\}$ in Line 7 to make the acceptance condition easier to achieve in the early stage when $\norm{\tilde p_t}$ is still large, and our analysis below still remains valid.)

\subsection{Analysis}

We show first that the line search criterion in Algorithm~\ref{alg:newton2} is satisfied for all $\alpha_t$ sufficiently small (\cref{lemma:ls}), and then deduce  global convergence (\cref{lemma:global}).
The first lemma does not require convexity of $f$; the second requires the property \eqref{eq:B}, which holds only if $\nabla^2 f$ is positive semidefinite at stationary accumulation points of the sequence $\{ x_t \}$.
From standard analysis of proximal gradient, we know that 
\begin{equation}
F(\bar x_{t+1}(\alpha_t)) \leq F(y_{t+1}(\alpha_t)),\quad \forall
\alpha_t \geq 0,
\label{eq:pg}
\end{equation}
where $\bar x_{t+1}$ is defined in \cref{eq:xbar}.
This fact will play an important role in the following two results.
\begin{lemma} \label{lemma:ls}
Given $\beta \in (0,1)$ and $\gamma \in (0,1)$, assume that $f$ is $L$-Lipschitz-continuously differentiable for some $L > 0$ and that $\Psi$ is convex, proper, and closed.
At iteration $t$ of Algorithm~\ref{alg:newton2}, the final value of $\alpha_t$ satisfying the sufficient decrease condition \eqref{eq:newarmijo} satisfies the following condition:
\begin{equation}
\mbox{either $\alpha_t=1$} \quad\mbox{ or }\quad
\alpha_t \ge \beta\mu_t\big( L+
2\gamma\norm{\tilde p_t}^{\delta} \big)^{-1}.
\label{eq:lsbound}
\end{equation}
\end{lemma}
\begin{proof}
Since $f$ is $L$-Lipschitz-continuously differentiable, we have 
\begin{equation*}
f\left( x_t + \alpha_t \tilde p_t  \right) \leq f\left( x_t \right) +
\alpha_t\inprod{\nabla f(x_t)}{\tilde p_t} +
\frac{\alpha_t^2 L}{2} \norm{\tilde p_t}^2.
\end{equation*}
Convexity of $\Psi$ implies that
\begin{equation*}
	\Psi ( x_t + \alpha_t \tilde p_t ) - \Psi ( x_t )\leq
	\alpha_t \left( \Psi ( x_t + \tilde p_t ) - \Psi ( x_t
	\right) ) = \alpha_t \left( \Psi\left( \tilde{x}_{t+1} \right) - \Psi\left( x_t
	\right) \right).
\end{equation*}
By summing these two inequalities, we obtain
\begin{eqnarray}
\nonumber
F\left( x_t + \alpha_t \tilde p_t \right)
&\le& F\left( x_t \right) + \alpha_t \left(\inprod{\nabla f(x_t)}{\tilde p_t} +
\Psi\left( \tilde x_{t+1} \right) - \Psi\left( x_t \right)\right) +
\frac{\alpha_t^2 L}{2} \norm{\tilde p_t}^2\\
&\stackrel{\cref{eq:stop}}{\leq}&
F\left( x_t \right) - \frac{\alpha_t}{2} \inprod{\tilde p_t}{H_t \tilde p_t} +
\frac{\alpha_t^2 L}{2} \norm{\tilde p_t}^2.
	\label{eq:lsb1}
\end{eqnarray}
From  $J_t \succeq 0$, we have that $H_t \succeq \mu_t I$ and thus $-\inprod{\tilde p_t}{H_t \tilde p_t} \le -\mu_t\norm{\tilde p_t}^2$.
Therefore, \eqref{eq:lsb1} and \cref{eq:pg} lead to
\begin{equation*}
F\left( \bar x_{t+1}\left(\alpha_t\right) \right) - F\left( x_t
\right) \leq F\left( x_t + \alpha_t \tilde p_t \right) - F\left( x_t
\right) \le \frac{\norm{\tilde p_t}^2}{2}\left( \alpha_t^2 L - \alpha_t
\mu_t \right),
\end{equation*}
implying that \cref{eq:newarmijo} is satisfied provided that
\begin{equation*}
\frac{\norm{\tilde p_t}^2}{2}\left( \alpha_t^2 L - \alpha_t \mu_t
\right) \le -\alpha_t^2\gamma \norm{\tilde p_t}^{2+\delta} \quad \stackrel{\alpha_t,
	\norm{\tilde p_t} \geq 0}{\Longleftrightarrow} \quad
\alpha_t \le \mu_t \big( L+ 2 \gamma\norm{\tilde p_t}^\delta \big)^{-1}.
\end{equation*}
The bound \cref{eq:lsbound} is then obtained after considering the overshoot of backtracking.
\end{proof}

In the following result, we denote by $\cS$ the set of stationary points for \cref{eq:f}, that is, points $x$ for which $r(x)=0$. Such points are solutions of \cref{eq:fmono}.
This result additionally requires $\theta \leq 1$, but
we note that this condition is more flexible than existing results such as those in \cite{MorYZZ20a}, which require $\theta = 1$.
The condition $\theta \leq 1$ together with  \cref{eq:B} imply that we also need $\rho \leq 1$.
From Item 3 of \cref{ex:ratesmono}, we see that when $q=1$ in \cref{eq:qmetricmono} (that is, when the Luo-Tseng error bound holds), since $p \leq 1$, then $\rho \leq 1$ is already a necessary condition for superlinear convergence.
Thus, our requirement is not more restrictive than existing conditions in the literature.
In the proof of the following lemma, we use a technical result originally from \cite[(12)]{YueZS19a} to upper bound $r_t$ by $\norm{\tilde p_t}$.
For completeness, we state and prove it in \cref{app:lemma}.
\begin{lemma}
\label{lemma:global}
Consider the setting of \cref{lemma:ls} and assume that $\theta \leq 1$, $\delta \leq 2$, $F$ is lower bounded by some $\bar F > -\infty$, and $J_t$ satisfies \cref{eq:B} for all $t$.
Then either \cref{alg:newton2} terminates at some $x_t$ for which $r(x_t)=0$ (thus $x_t \in \cS$), or else
\begin{equation}
	\lim_{t\rightarrow \infty} r(x_t) = 0.
	\label{eq:globalconv}
\end{equation}
If in addition $\cS$ is nonempty, then any accumulation point of $\{x_t\}$ is in $\cS$.
\end{lemma}
\begin{proof}
In the case of finite termination, it follows from the condition \cref{eq:stop} that $\tilde{x}_{t+1}=x_t$ only if $r(x_t)=0$. 
Otherwise, by summing \cref{eq:newarmijo} from $t=0$ to $t = \infty$ with the understanding that $x_{t+1} = \bar x_{t+1}(\alpha_t)$, and by telescoping, we have that
\begin{equation} \label{eq:sh1}
	\gamma \sum_{t=0}^\infty \alpha_t^2 \norm{\tilde p_t}^{2+\delta} \leq F\left( x_0
	\right) - \bar{F} \quad \Rightarrow \quad \lim_{t \rightarrow \infty} \,
	\alpha_t^2 \norm{\tilde p_t}^{2+\delta} = 0.
\end{equation}
Let $\cK_1$ be the subsequence of iterates $t$ for which $\alpha_t=1$ and $\cK_2$ be the complementary set, for which $\alpha_t \ge \beta\mu_t\big( L+
2\gamma\norm{\tilde p_t}^{\delta} \big)^{-1}$, by \cref{lemma:ls}.
We further partition $\cK_2$ as $\cK_2 = \cK_{2a} \cup \cK_{2b}$, where for $t \in \cK_{2a}$ we have $L/(2\gamma) \geq \norm{p_{t}}^{\delta}$, while for $t \in \cK_{2b}$ we have $L/(2\gamma) < \norm{p_{t}}^{\delta}$.

If $\cK_1$ is infinite, we have from \cref{eq:sh1} that 
\begin{equation}
    \label{eq:sh2}
    \lim_{t \to \infty, \, t \in \cK_1} \, \norm{\tilde p_t} = 0.
\end{equation}
From our assumption that $f$ is $L$-Lipschitz continuous, we have that $\norm{\nabla^2 f} \leq L$, which together with \cref{eq:subprob,eq:B} indicates that
\begin{equation*}
\norm{H_t} \leq L + \mu_t + O(r_t^\theta).
\end{equation*}
Combination of this inequality with \cref{lem:rtbound} then implies that
\begin{equation}
	(1 - \nu) r_t \leq (\|H_t\| + 2) \norm{\tilde p_t} \leq (L + \mu_t + O(r_t^\theta) + 2) \norm{\tilde p_t}.
	\label{eq:rpbound}
\end{equation}
Assume for contradiction that $\{r_t\}_{t \in \cK_1}$ is not upper bounded,
then there is a subsequence $\cK_1^r \subset \cK_1$ such that 
\begin{equation} \label{eq:sh4}
\lim_{t \rightarrow \infty, t \in \cK_1^r}\, r_t = \infty.
\end{equation}
The equation above together with \cref{eq:rpbound}, the definition of $\mu_t$ in \cref{eq:subprob}, and the condition  $\theta \ge \rho$ in \cref{eq:B} implies that there is a constant $C \geq 0$ such that 
\begin{equation} \label{eq:sh5}
	r_t \leq C r_t^{\theta} \|\tilde p_t\|, \;\; \forall t \in \cK_1^r.
\end{equation}
The conditions \cref{eq:sh4} and \cref{eq:sh5} together imply that 
\begin{alignat*}{2}
	\infty &= \lim_{t \rightarrow \infty, t \in \cK_1^r} r_t^{1 - \theta}
	\leq \lim_{t \rightarrow \infty, t \in \cK_1^r} C \|\tilde p_t\|,
	&& \quad \text{ if $\theta < 1$},\\
	1 & \leq \lim_{t \rightarrow \infty, t \in \cK_1^r} C
	\|\tilde p_t\|,
	&& \quad \text{ if $\theta = 1$},
\end{alignat*}
a contradiction with \cref{eq:sh2}.
Therefore, we know that $\{r_t\}_{t \in \cK_1}$ is upper bounded, and \cref{eq:sh2} and \cref{eq:rpbound} thus imply that
\begin{equation}
    \lim_{t \rightarrow \infty, t \in \cK_1} r_{t} = 0.
    \label{eq:result0}
\end{equation}

Now we turn to the situation of  $\cK_2$  infinite.
From \cref{eq:sh1}, we have
\begin{equation}
\lim_{t \rightarrow \infty, t \in \cK_2} \norm{\tilde p_t}^{2+\delta}\frac{\mu_t^2}{4}
\left( \frac{L}{2}+ \gamma\norm{\tilde p_t}^{\delta} \right)^{-2} = 0.
\label{eq:inf}
\end{equation}
When the subsequence $\cK_{2a}$ is infinite, we have from \cref{eq:inf} that
\begin{equation}
	\lim_{t \rightarrow \infty, t \in \cK_{2a}} \norm{p_{t}}^{2+\delta}\frac{\mu_{t}^2}{4L^2}
= 0 \quad \Rightarrow \quad
\lim_{t \rightarrow \infty, t \in \cK_{2a}} \norm{p_{t}}^{2+\delta}\mu_{t}^2 = 0.
\label{eq:case1}
\end{equation}
By applying \cref{eq:rpbound} to \cref{eq:case1}, we then obtain
\begin{equation}
	\lim_{t \rightarrow \infty, t \in \cK_{2a}} \frac{\mu_{t}^2 r_{t}^{2+\delta}}{\left(2 + L +
\mu_{t} +
O\left( r_{t}^\theta \right)\right)^{2+\delta} } = 0.
\label{eq:penultimate}
\end{equation}

By inserting the definition \cref{eq:subprob}  of $\mu_t$ into \cref{eq:penultimate}, we get
\begin{equation}
	\lim_{t \rightarrow \infty, t \in \cK_{2a}} \frac{c^2 r_{t}^{2\rho + 2+\delta}}{\left(2 + L +
c r_t^\rho +
O\left( r_{t}^\theta \right)\right)^{2+\delta} } = 0.
\label{eq:penultimate2}
\end{equation}
If the sequence $\{ r_t\}_{t \rightarrow \infty, t \in \cK_{2a}}$ has an accumulation point at some positive finite value, the limit \cref{eq:penultimate2} cannot hold.
If there is a subsequence approaching $\infty$, then because $1 \geq \theta \geq \rho$, the denominator in  \cref{eq:penultimate2} is $O(r_t^{2+\delta})$ and since the numerator is $c^2r_t^{2\rho+2+\delta}$, so the limit in \cref{eq:penultimate2} over this subsequence is infinite, contradicting\cref{eq:penultimate2}. 
We therefore conclude that
\begin{equation}
\lim_{t \rightarrow \infty, t \in \cK_{2a}} r_{t} = 0.
\label{eq:result1}
\end{equation}
(To double check, we see that when \cref{eq:result1} holds,
$(2 + L + c r_t^\rho + O(r_t^\theta)) = O(1)$, and it indeed leads to \cref{eq:penultimate2} and thus \cref{eq:penultimate}.)

Suppose next that the subsequence $\cK_{2b}$ is infinite.
If $\delta = 2$, we have from  \cref{eq:inf} that
\begin{equation}
	\lim_{t\rightarrow \infty, t \in \cK_{2b}} \mu_{t}^2 = 0\quad
	\stackrel{\cref{eq:subprob}}{\Rightarrow} \quad
\lim_{t \rightarrow \infty, t \in \cK_{2b}} r_{t} = 0.
\label{eq:result2}
\end{equation}
For $\delta < 2$, from the lower-boundedness of $\|\tilde p_t\|$ in $\cK_{2b}$, we see that
\begin{equation}
\label{eq:result3}
\lim_{t \rightarrow \infty, t \in \cK_{2b}} \norm{\tilde p_t}^{2 - \delta} \mu_t^2 = 0 \quad \Rightarrow \lim_{t \rightarrow \infty, t \in \cK_{2b}} \mu_t^2 = 0 \quad \Rightarrow \lim_{t \rightarrow \infty, t \in \cK_{2b}} r_t = 0.
\end{equation}

The desired result for the full sequence $\{r_t\}$ is thus obtained by combining \cref{eq:result0,eq:result1,eq:result2,eq:result3}.
The claim that accumulation points are in $\cS$  follows from the continuity of $r(\cdot)$.
\end{proof}

The previous two lemmas require no H\"older continuity of $\nabla^2 f$ nor the H\"olderian error bound condition, and are therefore  valid even if the problem does not fall in the class that allows for
superlinear convergence in our previous analysis.

Next,  we show in \cref{lemma:unit2} that the unit
step $\alpha_t=1$ is eventually always accepted when the conditions of
\cref{thm:supermono} hold and $f$ is
convex, so that $Q$-superlinear convergence of $\{d_t\}$ and
$\{r_t\}$, defined in \cref{eq:def.dr}, is guaranteed, 
and the objective function converges superlinearly to its optimal value.
\cref{lemma:unit2} requires $\delta$ to be large enough such that $\|\tilde p_t\|^{2+\delta}$ is dominated by $d_t^{(q+1)/q}$ for all large $t$, or equivalently when $\tilde p_t$ is sufficiently small, where $\tilde p_t$ is the prox-Newton update step defined in \cref{eq:tildep}.
Noting the bound \cref{eq:pbmono}  (which, as noted above, applies to $\tilde p_t$ as well as to $p_t$), we need to select $\delta$ to satisfy
\begin{equation} \label{eq:so1}
(2+\delta) \min\Big\{ 1, 1 + p - \frac{\rho}{q} \Big\} > 1 + q^{-1}.
\end{equation}
It suffices for this inequality that the following two conditions hold:
\begin{subequations} \label{eq:delta}
    \begin{align} 
    \label{eq:delta.1}
        \delta &> q^{-1} - 1, \\
        \label{eq:delta.2}
	2+\delta &\geq \frac{q+1}{q}	(1+p)q = (1+p)(1+q).
    \end{align}
\end{subequations}
To confirm that \eqref{eq:delta.2} suffices, we have from \eqref{eq:ineqmono} that 
\[
(2+\delta) \left(1+p- \frac{\rho}{q}\right)  \ge (1+p)(1+q) \frac{1+pq-\rho}{q} >\frac{1+q}{q}.
\]
When $q \leq 1$, since $p  \in (0,1]$, the condition  \cref{eq:delta.2} is satisfied as long as $\delta \geq 2$.
For condition \eqref{eq:delta.1}, we see from \cref{ex:ratesmono} that when $p=1$ we have $q^{-1} < 8 / (-1+\sqrt{33})< 1.7$. (If $p < 1$, the lower bound for $q$ is larger, making $q^{-1}$ even smaller.)
Thus, it suffices for \eqref{eq:delta.1} that $\delta \geq 0.7$.
Setting $\delta = 2$ thus works for both conditions in \eqref{eq:delta}.

\begin{theorem}
Consider \cref{eq:f} and assume that the settings of \cref{thm:supermono} hold, with $A(x) = \nabla f(x)$ and $B(x) = \partial \Psi(x)$, and in particular that the quantities $\rho \ge 0$, $p \in (0,1]$ and $q > 0$ satisfy \cref{eq:ineqmono}.
Assume too that the settings of \cref{lemma:ls} hold and $f$ is convex.
Consider \cref{alg:newton2} and let $F^* \coloneqq \min F(x)$.
If $\delta>0$ satisfies $\norm{\tilde p_t}^{2+\delta} = o(d_t^{(q+1)/q})$, for all $t$ sufficiently large,
then there is $t_0 \geq 0$ such that $\alpha_t = 1$ is accepted for all $t \geq t_0$ and, for $s>0$ defined as in \cref{thm:supermono}, we have
\begin{equation}
\left\{
	\begin{aligned}
		d_{t+1} &= O\left( d_t^{1+s} \right),\\
		r_{t+1} &= O\left( r_t^{1+s} \right),\\
		F\left(x_{t+1}\right) - F^* &= O\big( \left(F\left( x_{t+1} \right)-
		F^*\right)^{1+s} \big),
	\end{aligned}
 \right.
\quad \forall t \geq t_0.
\label{eq:super2}
\end{equation}
\label{lemma:unit2}
\end{theorem}
\begin{proof}
For purposes of this proof, we use the abbreviations 
\[
\bar x_{t+1} \coloneqq \bar x_{t+1}(1), \quad y_{t+1} \coloneqq y_{t+1}(1).
\]
We first list two auxiliary results.
From \cite[Lemma~2.3]{BecT09a}, we have that
\begin{equation}
	F\Big( x -  \tfrac{1}{L} G_{L}(x) \Big) - F^* \leq
	\norm{G_{L}(x)}^2 \left(
	\frac{\dist(x,\cS)}{\norm{G_{L}( x)}} -
	\frac{1}{2{L}} \right), \quad \forall x \in \H.
\label{eq:pgbound}
\end{equation}
Moreover, \cite[Lemma~3]{TseY09a} indicates that there are constants $C_1 \geq
C_2 > 0$ such that $r(x)$ and the norm of $G_L(x)$ are related by
\begin{equation}
	C_1 r(x) \geq \norm{G_{L}(x)} \geq C_2 r(x), \quad \forall x.
\label{eq:RL}
\end{equation}
Now, using \cref{eq:pgbound} with
	$x=y_{t+1}$ and
$\bar{x}_{t+1} =
y_{t+1} - \tfrac{1}{L} G_{{L}}(y_{t+1})$
as in \cref{eq:xbar},
we have that
\begin{eqnarray}
\nonumber
F\left( \bar x_{t+1} \right) - F^* &\leq& \norm{G_{L}(y_{t+1})}
	\dist\left( y_{t+1}, \cS \right)\\
\nonumber
&\stackrel{\cref{eq:RL}}{\leq}& C_1 r(y_{t+1}) \dist\left( y_{t+1},
	\cS \right)\\
&\stackrel{\cref{eq:qmetricmono}}{\le}& C_1 \kappa r\left( y_{t+1}
	\right)^{1+q}.
\label{eq:ub}
\end{eqnarray}
When $\alpha_t = 1$, we have $y_{t+1} = \tilde x_{t+1}$.
Therefore, we can apply \cref{eq:finalformmono} to \eqref{eq:ub} and obtain
\begin{equation}
F( \bar x_{t+1} ) - F^* \le
O\Big( d_t^{(1+q)\min\left\{ 1+\rho, 1+p,\left( 1 + p -
	\frac{\rho}{q} \right)(1+p) \right\}} \Big) = O\big( d_t^{(1+s)(1+q)/q} \big),
\label{eq:ub2}
\end{equation}
where we used the definition of $s$ from \cref{thm:supermono}.
On the other hand, \cref{eq:rtosubgrad} (due to convexity of $f$) and \cref{eq:qmetricmono} imply that there is $\epsilon_2 > 0$ such that
\begin{equation*}
\dist\left(x,\cS\right) \leq \kappa \norm{z}^q, \quad
\forall z \in \partial F(x),\quad \forall x \in  \left\{x\mid
\dist\left(x , \cS\right) \le \epsilon_2 \right\},
\end{equation*}
which together with convexity of $F$ implies that
\begin{equation}
F(x) - F^* \leq \kappa \min_{z \in \partial F(x)} \norm{z}^{1+q},
\quad \forall x \in  \left\{x\mid \dist\left(x , \cS\right) \le
\epsilon_2 \right\}.
\label{eq:subdiff}
\end{equation}
We then apply \cite[Theorem~3.4]{MorO15a} to \cref{eq:subdiff} to
conclude that there is $\kappa_2 > 0$ such that
\begin{equation}
	F(x) - F^* \geq \kappa_2 \dist(x, \cS)^{\frac{1+q}{q}}
	\label{eq:sharp}
\end{equation}
holds for all $x$ close enough to $\cS$.
Since $\norm{\tilde p_t}^{2+\delta} = o(d_t^{(1+q)/q})$, we have from
\cref{eq:sharp} that
\begin{equation}
	F\left( x_t \right) - F^* - \sigma \norm{\tilde p_t}^{2+\delta} \geq
	\Omega\Big(d_t^{\frac{1+q}{q}}\Big) + o\Big(
	d_t^{\frac{1+q}{q}} \Big) = \Omega\Big(d_t^{\frac{1+q}{q}}\Big).
\label{eq:lb}
\end{equation}
By comparing \cref{eq:ub2} and \cref{eq:lb}, and using the fact that \cref{eq:ineqmono} implies  $s>0$, we have for $d_t$ sufficiently small that
\begin{equation*}
	F\left( \bar x_{t+1} \right) - F^* = o\big( F(x_t) - F^* - \sigma
	\norm{\tilde p_t}^{2+\delta}\big).
\end{equation*}
Therefore, from \cref{eq:globalconv} and \cref{eq:qmetricmono}, we have for all sufficiently large $t$ that
\begin{equation*}
	F\left( \bar x_{t+1} \right) - F(x_t) + \sigma
	\norm{\tilde p_t}^{2+\delta} =
	F( \bar x_{t+1} ) - F^* - \big(F(x_t) - F^* - \sigma
	\norm{\tilde p_t}^{2+\delta}\big) \le 0,
\end{equation*}
meaning that the unit step size is accepted and we have $x_{t+1} = \bar{x}_{t+1}$.
From convexity of $f$ and \cite[(34)]{LeeW19a}, a short proximal-gradient step does not increase distance to the solution, so we have  $d_{t+1} =\dist\left(x_{t+1},\cS\right) = \dist\left( \bar
x_{t+1}, \cS \right) \leq \dist\left( y_{t+1},\cS \right)$. 
We can now apply \cref{eq:superconvmono}, noting that $d_{t+1}$ in that bound corresponds to $\dist(y_{t+1},\cS)$ here because
$\alpha_t = 1$ (while our $d_{t+1}$ in this section corresponds to $\dist(x_{t+1}, \cS)$), to obtain
\begin{equation}
	\underbrace{\dist(x_{t+1}, \cS)}_{d_{t+1} \text{ in this section}} \leq \underbrace{\dist(y_{t+1}, \cS)}_{d_{t+1} \text{in \cref{eq:superconvmono}}} = O(d_t^{1+s}) = o(d_t).
    \label{eq:d_decrease}
\end{equation}
This proves the first claim in \cref{eq:super2}.
When $d_t$ is small enough, we further have from \cref{eq:d_decrease} that $d_{t+1} \leq d_t$ so that  superlinear
convergence and acceptance of the unit step propagates to the next
iteration.
We note that there must be $t$ such that $d_t$ is small enough to satisfy our requirement according to \cref{lemma:global} and
\cref{eq:qmetricmono}.

Superlinear convergence of the objective follows from \cref{eq:sharp} and \cref{eq:ub2}. 
In particular, using $x_{t+1}=\bar{x}_{t+1}$ in \cref{eq:ub2}, we see that
\[
F(x_{t+1}) - F^* \leq  O\left( d_t^{(1+s) (1+q) / q} \right)
	\stackrel{\cref{eq:sharp}}{=} O\big(\big( F(x_t) -
	F^*\big)^{1+s} \big).
\]

Finally, to prove the convergence rate for $r_t$, we use the definition  \cref{eq:defR3} and  nonexpansiveness of $\prox_\Psi$ to obtain
\begin{align}
\nonumber
&~ \norm{R\left( x_{t+1} \right) - R\left( y_{t+1} \right)} \\
\nonumber
=&~ \norm{(x_{t+1}-y_{t+1}) - \left(\prox_{\Psi}(x_{t+1} - \nabla f(x_{t+1})) - \prox_{\Psi}(y_{t+1} - \nabla f(y_{t+1})) \right)} \\
\nonumber
\leq&~ \norm{(x_{t+1} - y_{t+1}}  + \norm{(x_{t+1} - \nabla f(x_{t+1}))-(y_{t+1} - \nabla f(y_{t+1})) } \\
\leq&~ (2+L)\norm{x_{t+1} - y_{t+1}}.
\label{eq:rtobound}
\end{align}
Using  from \cref{eq:xbar} that $x_{t+1} = \bar{x}_{t+1} = y_{t+1} - (1/{L}) G_{{L}} (y_{t+1})$, and using \cref{eq:RL},
we can bound $\norm{x_{t+1} - y_{t+1}}$ by
\begin{equation*}
	\norm{x_{t+1} - y_{t+1}} = \frac{1}{{L}}
	\norm{G_{{L}}(y_{t+1})}
	\in \left[\frac{C_2}{{L}} r\left(y_{t+1}\right),
		\frac{C_1}{{L}}
	r\left(y_{t+1}\right)\right],
\end{equation*}
where $C_1 \ge C_2 >0$.
By substituting  into \cref{eq:rtobound}, we obtain
\begin{align*}
	r_{t+1} &= \norm{R(x_{t+1})} \leq r( y_{t+1}) + (2+L) \norm{x_{t+1} - y_{t+1}}\\
	&\leq \Big(1 + \frac{(2+L)C_1}{{L}}\Big)r( y_{t+1} )\\
	&= O\left( r(y_{t+1}) \right) \\
	&= O\left( r_t^{1+s} \right),
\end{align*}
(where the last step is from \cref{eq:superconvmono} in \cref{thm:supermono}), proving \cref{eq:super2}.
With the same argument for $d_t$, we see that \cref{eq:qmetricmono}
continues to hold for the next iterate when $r_t$ is small enough to
ensure $r_{t+1} \leq r_t$.
\end{proof}

Finally, we show that when $q=1$ in \cref{eq:qmetricmono}, our linesearch is versatile enough to accept the unit step size for \emph{any} update directions that yield $Q$-superlinear convergence of both $d_t$ and $r_t$, regardless of how these directions are generated.
\begin{corollary}
    \label{cor:unit3}
    Consider \cref{eq:f} and
assume that the settings of \cref{lemma:ls} hold, $f$ is convex, and \cref{eq:qmetricmono} holds with $q=1$.
Consider \cref{alg:newton2} but with $\{\tilde x_{t+1}\}$ generated in 
an arbitrary manner that satisfies
\begin{equation}
        \dist(\tilde x_{t+1}, \cS) = o(d_t) \quad \text{and} \quad r(\tilde x_{t+1}) = o(r_t).
        \label{eq:fast_cand}
    \end{equation}
    Assume too that the parameter $\delta > 0$ in \cref{alg:newton2} satisfies $\|\tilde p_t\|^{2+\delta} = o(d_t^2)$.
Then for all $t$ sufficiently large, we have $\alpha_t = 1$, and we have
\[
d_{t+1} = o(d_t),\quad r_{t+1} = o(r_t),\quad F(x_{t+1}) - F^* = o(F(x_t) - F^*).
\]
\end{corollary}
\begin{proof}
The proof mainly follows the argument for \cref{lemma:unit2} with a few differences noted here.
First, due to \cref{eq:fast_cand} and $q=1$, and noting that $r_t = O(d_t)$ from \eqref{eq:Lipmono}, \cref{eq:ub2} becomes
\[
F(\bar x_{t+1}) - F^* \leq o(d_t^2).
\]
On the other hand, our assumption that $\|\tilde{p}_t\|^{2+\delta} = o(d_t^2)$ and \cref{eq:sharp} with $q=1$ imply \cref{eq:lb} with the right-hand side being $\Omega(d_t^2)$.
By comparing with the inequality above, we conclude that the unit step size is accepted for $t$ sufficiently large.
The claim of superlinear convergence then follows by the same reasoning as in the remainder of the argument in  \cref{lemma:unit2}, with $d_t^{1+s}$, $(F(x_t) - F^*)^{1+s})$, and $r_t^{1+s}$ replaced by $o(d_t)$,  $o(F(x_t) - F^*)$, and $o(r_t)$, respectively.
\end{proof}

From the proofs of \cref{lemma:unit2,cor:unit3}, we can see that
the analysis relies solely on the convergence rate of the update direction rather than the existence or continuity of the Hessian  $\nabla^2 f$.
Therefore, our novel line search serves as a general framework that is also compatible with any  ``fast directions'' obtained beyond the proximal Newton approach analyzed in \cref{thm:supermono,cor:cont}, such as through a (proximal) semismooth Newton or a proximal quasi-Newton approach, to simultaneously ensure strict objective decrease and eventual unit step size acceptance for fast convergence.

\section{Simplification For Smooth Problems}
\label{sec:simple}

Our algorithm can be simplified for the case of convex smooth optimization --- the setting of \Cref{sec:global} with  $\Psi \equiv 0$, that is, $F(x) = f(x)$.
We have in this scenario that $R(x) = \nabla f(x)$ and $G_L(x) = \nabla f(x)$ for any $L > 0$, and the bound \cref{eq:pgbound} simplifies to
\begin{equation}
F(x) - F^* \leq \norm{\nabla f(x)}\dist(x, \cS)
\label{eq:easy}
\end{equation}
from convexity.
Therefore, the proximal gradient step can be removed from \cref{alg:newton2} without affecting the bounds on the objective value.
The simplified method is shown as \cref{alg:newton1}.

\begin{algorithm}[tbh]
\DontPrintSemicolon
\SetKwInOut{Input}{input}\SetKwInOut{Output}{output}
\caption{A Simple Newton Method for Degenerate Problems}
\label{alg:newton1}
\Input{$x_0\in\H$, $\beta,\gamma\in (0,1)$, $\nu \in [0,1)$, $c
> 0, \rho \in (0,1]$, $\delta \geq 0$}

\For{$t=0,1,\ldots$}{
Select $H_t$ satisfying \cref{eq:subprob,eq:B} and find an approximate solution $\tilde x_{t+1}$
of \cref{eq:newton} (with $\Psi = 0$) satisfying \cref{eq:stop}

	$\tilde p_t \leftarrow \tilde x_{t+1} - x_t$, $\alpha_t \leftarrow 1$

	\While{ $F(x_t + \alpha_t \tilde p_t) > F(x_t) - \gamma
	\alpha_t^2 \norm{\tilde p_t}^{2+\delta}$}
	{
			$\alpha_t \leftarrow \beta \alpha_t$
	}
	
	$x_{t+1} \leftarrow x_t + \alpha_t \tilde p_t$
}
\end{algorithm}

Clearly, this is identical to the classical truncated Newton method except for the addition of a damping term to the quadratic approximation and the slightly unconventional step size acceptance criterion.
As the proofs of \cref{lemma:ls,lemma:global} do not involve any specific properties of the proximal gradient step, we see that they are still applicable to \cref{alg:newton1}.
The local convergence result is as follows.
\begin{corollary}
Assume that $\Psi \equiv 0$ and $f \in \mathcal{C}^2$ is convex and $L$-Lipschitz-continuously
differentiable in \cref{eq:f} with $\nabla^2 f$
$p$-H\"older continuous in a neighborhood $U$ of $\cS$, and within $U$, \cref{eq:qmetricmono} holds for some
$\kappa > 0$ and some $q \in (0,1]$.
Then for \cref{alg:newton1}, we have that \cref{eq:lsbound} and \cref{eq:globalconv} both hold.
Further, if \cref{eq:ineqmono} is satisfied, then there is $t_0 \geq 0$ such that
$\alpha_t = 1$ for all $t \geq t_0$, and \cref{eq:super2} holds.
\label{cor:smooth}
\end{corollary}
\begin{proof}
The claims for  \cref{eq:lsbound} and \cref{eq:globalconv} follow directly from the same reasoning as before, noting that $x_{t+1} = y_{t+1}$.
For the superlinear convergence claim, we see from \cref{eq:qmetricmono} (recalling that $R(x) = \nabla f(x)$ in this case) and \cref{eq:easy} that the following condition holds.
\begin{equation}
	F(x) - F^* \leq \kappa \norm{\g}^{1+q}.
	\label{eq:KL}
\end{equation}
Theorem~3.4 of \cite{MorO15a} then indicates that \cref{eq:sharp} holds for some $\kappa_2 > 0$ near $\cS$.
Finally, using the argument in the proof of \cref{lemma:unit2}, with \cref{eq:pgbound} replaced by \cref{eq:easy} for $x_{t+1} = y_{t+1}$,
the desired results in \cref{eq:super2} follow.
\end{proof}

Recent works \cite{Mis21a,DoiN21a} consider a specific damping in the form of \cref{eq:subprob} with $\rho = 1/2$ and $J_t = \nabla^2 f(x_t)$ for unconstrained convex optimization.
Under the assumption that $\nabla^2 f$ is globally $H$-Lipschitz continuous, these two works showed that by fixing $c$ to a specific value in \cref{eq:subprob} and solving the subproblem exactly, $\alpha_t \equiv 1$ can be used and global convergence of the objective value to the optimum is guaranteed, with a speed of $O(t^{-2})$.
They also showed local superlinear convergence under an additional global strong convexity condition (global strong convexity, continuous Hessian, and a descent algorithm imply that the gradient is Lipschitz continuous in the region of interest).
However, we notice that their analyses utilized the second-order Taylor approximation of $f$ and the Lipschitz continuity of the Hessian, and thus are not applicable to the problem we consider here, where we do not assume Lipschitz continuity of  $\nabla^2 f$.
Therefore, \cref{cor:smooth} extends the problem class on which their damped truncated Newton leads to superlinear convergence to non-strongly-convex ones with non-Lipschitz Hessians.

\section{Conclusions}
We have examined inexact, damped, proximal Newton-like methods for solving degenerate regularized optimization problems, and their extension to the generalized equations setting.
We have described algorithms that achieved superlinear convergence even in the presence of nonconvexity of the smooth term and singularity of its Hessian (or nonmonotonicity, and singularity of the Jacobian, in the case of generalized equations).
Moreover, we show that the standard assumptions of Lipschitz
continuity of the Hessian (or Jacobian) and the Lipschitz error bound can be relaxed to H\"olderian
\ifdefined\arxiv
ones,
and we can further relax the H\"older continuity assumption of the
Hessian (or Jacobian) to uniform continuity.
\else
ones.
\fi
These results require careful choices of the parameters that govern the damping and the measure of inexactness in the solution of each subproblem.

\ifdefined\arxiv
\else
We can show, in a result not included here (see the arXiv version of
this paper for details in \cite[Appendix~A]{LeeW26a}), that the H\"older
continuity assumptions of
the Hessian (or Jacobian) can
be relaxed further to uniform continuity, provided certain additional
conditions on the choice of parameters and the inexactness conditions
are satisfied.
\fi

\section*{Acknowledgements}
We thank Defeng Sun for fruitful discussion and pointing out an error in an early draft.

\bibliographystyle{siamplain}
\bibliography{header-arxiv,degeneratePN}
\appendix

\ifdefined\arxiv
\section{Relaxation from $p$-H\"older Continuity of the Hessian to
Uniformly Continuous}
\label{app:unicont}
This appendix shows that the $p$-H\"older assumption
\cref{eq:DAholder} on $\nabla A$ in \cref{sec:local} can be further
relaxed.
In particular, for the case of $q=1$, superlinear convergence can be
preserved even when the Jacobian is merely \emph{uniformly
continuous}, provided that the damping and the stopping tolerance
decay slowly enough.
For this result, we note that for any function $f$ uniformly
continuous in a convex set $V$, it admits a modulus of continuity
$\omega: [0,\infty) \rightarrow [0,\infty)$ such that
\begin{equation}
    \lim_{s \downarrow 0}\, \omega(s) = \omega(0) = 0,\quad
    \norm{f(x) - f(y)} \leq \omega(\norm{x-y}),\quad \forall x,y \in V.
    \label{eq:modulus}
\end{equation}
It is known that we can always select an $\omega$ that is
continuous, monotonically increasing, and subadditive.
See, for example, \cite[Chapter~2, Section~6]{DevL93a}.

Subadditivity and monotonicity of $\omega$ play a crucial role in our analysis below.
More specifically, subadditivity and monotonicity of $\omega$ imply that
\begin{equation}
	f(t) = O(g(t)) \quad \Rightarrow \quad
	\omega(f(t)) = O(\omega(g(t)).
	\label{eq:subadd}
\end{equation}
Indeed,
by definition, $f(t) = O(g(t))$ means there is $\beta > 0$ such
that $f(t) \leq \beta g(t)$ for all $t$ large, and thus monotonicity of
$\omega$ implies that
\[
	\omega(f(t)) \leq \omega(\beta g(t)) \leq
\omega(\lceil \beta \rceil g(t)),
\]
while subadditivity of $\omega$
indicates
\[
\omega(\lceil \beta \rceil g(t))\leq \lceil \beta \rceil
\omega(g(t)).
\]
These two inequalities in combination then lead to
\[
\omega(f(t)) \leq \lceil \beta \rceil \omega(g(t)) =
O(\omega(g(t)),
\]
which is exactly \cref{eq:subadd}.
\begin{corollary}
\label{cor:cont}
Consider \cref{eq:fmono} with the same assumptions as
\cref{thm:supermono}, except that \cref{eq:qmetricmono} holds with $q
= 1$ and $\nabla A$ is only uniformly continuous in $V$.
Let $\omega_1$ denote the nondecreasing and subadditive modulus of continuity of $\nabla A$.
Consider the update scheme \cref{eq:local}, but with $\mu_t$, $J_t$, and the stopping condition \cref{eq:stopmono} replaced by
\begin{equation}
    \mu_t = c\, \omega_2(r_t),\quad \norm{J_t - \nabla A(x_t)} = O(\omega_2(r_t)), \quad \hat r_t(x_{t+1}) \leq \nu \omega_2(r_t)r_t
    \label{eq:damp}
\end{equation}
for some $\nu \geq 0$ and some given nondecreasing, subadditive, and continuous function $\omega_2: [0,\infty)
	\rightarrow [0,\infty)$ that vanishes at zero.
If the problem satisfies
\begin{equation}
\omega_1(s) \leq \beta_1 \omega_2(s)
\label{eq:omegas}   
\end{equation}
for some $\beta_1 > 0$ and all $s \geq 0$,
then, provided $r_0$ is sufficiently small, we have
$r_{t+1} = o(r_t)$, $d_{t+1} = o(d_t)$, and $\{x_t\}$ converges strongly to a point in the solution set $\cS$.
\end{corollary}
\begin{proof}
We use the same notations as \cref{eq:def.dr} such that
\[
d_t \coloneqq \dist(x_t,\cS),\quad r_t \coloneqq r(x_t),\quad p_t \coloneqq x_{t+1} - x_t,
\]
and $\bar x_t \in P_{\cS}(x_t)$, where $\cS$ is the set of solutions.
    We observe first from \cref{eq:modulus} that uniform continuity of
	$\nabla A$ implies that
     \begin{align}
		 \nonumber
         \norm{A(x) - A(y) - \nabla A(y)(x-y)}  & = 
     \norm{\int_0^1 \nabla A(y + t(x-y)) (x-y) \, dt - \nabla A(y) (x-y)}  \\
		 \nonumber
     & \le \int_0^1 \omega_1(t\|x-y\|) \, dt \, \norm{x-y}  \\
     & \le \omega_1(\norm{x-y})\norm{x-y}.
    \label{eq:unicont}
     \end{align}
    Second, since $q  = 1$, \cref{eq:qmetricmono,eq:Lipmono} indicate $r_t = O(d_t)$ and $d_t = O(r_t)$.
From the definitions \eqref{eq:residual} and \eqref{eq:rtmono}, we have $\| \xi_t \| = \hat{r}(x_{t+1})$ and so  \cref{eq:damp} implies that $\norm{\xi_t} \le \nu \omega_2(r_t) r_t$. 
Since $\omega_2(r_t) \to 0$ as $r_t \to 0$ and $r_t = O(d_t)$, we have
    \begin{equation}
        \norm{\xi_t} = o(r_t) = o(d_t).
        \label{eq:xi_bound}
    \end{equation}

    We now find a bound on  $\|\bar x_t - x_{t+1} + \xi_t\|$, using \eqref{eq:lower} again in a similar way to how we derived \eqref{eq:firsttermmono}. 
    Substituting the parameter choices from \cref{eq:damp} into \eqref{eq:lower}, and using \cref{eq:unicont} with $x=\bar x_t$ and $y=x_t$, we obtain: %
    \begin{align*}
        & \norm{\bar x_t - x_{t+1} + \xi_t} \\
        \le &~ \mu_t^{-1} \Big( \norm{(J_t - \nabla A(x_t))(\bar x_t - x_t)} + \norm{A(x_t) - A(\bar x_t) - \nabla A(x_t)(x_t - \bar x_t)} \\
        &~ \qquad + \mu_t \norm{\bar x_t - x_t} + (1 + \norm{H_t})\norm{\xi_t} \Big) \\
        \le &~ \frac{1}{c \omega_2(r_t)} \Big( O(\omega_2(r_t)) d_t + \omega_1(d_t) d_t + c \omega_2(r_t) d_t + O(1) \cdot \nu \omega_2(r_t) r_t \Big).
    \end{align*}
    In the last inequality, we used $\norm{H_t} = O(1)$, which is from 
    \begin{eqnarray}
    \label{eq:Hbound}
    \norm{H_t} &\stackrel{\cref{eq:local},\cref{eq:damp}}{=}
	&\norm{\mu_t I + \nabla A(x_t) + O(\omega_2(r_t))} \\
    \nonumber
    &\stackrel{\cref{eq:damp}}{\leq}& O(\omega_2(r_t)) + \norm{\nabla A(x_t)} \\
    \nonumber
    &\stackrel{\cref{eq:uppermono}}{\leq}& O(\omega_2(r_t)) + L 
    = O(1).
    \end{eqnarray}
      Using \cref{eq:omegas} and the fact that $d_t = O(r_t)$, we have
	  from \cref{eq:subadd} and monotonicity and subadditivity of
	  $\omega_2$ that $\omega_2(d_t) = O(\omega_2(r_t))$.
    Thus, the numerator is dominated by $O(\omega_2(r_t) d_t)$ since $r_t = O(d_t)$ according to \cref{eq:Lipmono}.
    Consequently, the damping term $\omega_2(r_t)$ in the denominator cancels out, yielding:
    \[
        \norm{\bar x_t - x_{t+1} + \xi_t} \le O(d_t).
    \]
    By combining the above inequality with \cref{eq:xi_bound}, and using the definition $\bar{x}_t = P_{\mathcal{S}}(x_t)$ and \eqref{eq:def.dr},  we obtain
    \begin{equation}
        \norm{p_t} \le \norm{\bar x_t - x_{t+1} + \xi_t} + \norm{x_t - \bar x_t} + \norm{\xi_t} = O(d_t) + d_t + o(d_t) = O(d_t).
        \label{eq:pd}
    \end{equation}
We now proceed to obtain an upper bound for $r_{t+1}$ similar to \cref{eq:gtmono} by following the proof of \cref{lemma:ytmono}.
    We know that \cref{eq:trimono,eq:prox1mono} still hold as they do not involve elements changed in this corollary, so they indicate
    \begin{eqnarray*}
    r_{t+1} &\leq& \norm{A(x_t) - A(x_{t+1}) - H_t (x_t - x_{t+1}}) + \hat r_{t} (x_{t+1})\\
    &\stackrel{\eqref{eq:Hbound},\cref{eq:damp}}{\leq} &\norm{A(x_t) - A(x_{t+1}) - \nabla A(x_t)(x_{t+1} - x_t)} + O(\omega_2(r_t)) \norm{p_t} + O(\omega_2(r_t) r_t)\\
    &\stackrel{\cref{eq:unicont}}{\leq} &
    \omega_1(\norm{p_t})\norm{p_t} + O(\omega_2(r_t) (\norm{p_t} + r_t)).
    \end{eqnarray*}
    From \cref{eq:pd,eq:qmetricmono}, we know that $\norm{p_t} =
	O(r_t)$, and from \cref{eq:omegas} and \cref{eq:subadd} resulted
	from subadditivity and monotonicity of $\omega_1$ and $\omega_2$,
	we conclude that $\omega_1(\norm{p_t}) \norm{p_t} =
	O(\omega_2(r_t) r_t)$.
    Therefore, the inequality above simplifies to
    \[
    r_{t+1} \leq O(\omega_2(r_t)r_t) = o(r_t)
    \]
    as $\omega_2$ vanishes at zero.
     Since $d_{t+1} = O(r_{t+1})$ and $r_t = O(d_t)$, the inequality above also shows that
    $d_{t+1} = o(d_t)$, proving the claimed superlinear convergence.
    Convergence of the sequence $\{x_t\}$ then follows from the argument of \cref{thm:iterate} using $\norm{p_t} = O(r_t)$.
\end{proof}
The case of $p$-H\"older continuity corresponds to \cref{cor:cont}
with $\omega_1(t) = \zeta t^p$.
However, for this special case,
\cref{thm:supermono,thm:supermono2} provide more refined results with
explicit rates.

Comparing \cref{cor:cont} with \cref{thm:supermono,thm:supermono2} reveals
an inherent trade-off in the decay rate of the damping term and the
stopping tolerance.
To ensure robustness against the lack of smoothness (small $p$ or
uniform continuity), these parameters must decay slowly.
Conversely, to accommodate a wider range of the error bound exponent,
they must vanish rapidly.
\fi

\section{A Technical Lemma}
\label{app:lemma}
\begin{lemma}
\label{lem:rtbound}
Consider the setting of \cref{lemma:global}.
For all $t \geq 0$, we have
\[
(1 - \nu) r(x_t) \leq (\norm{H_t} + 2) \norm{\tilde p_t}.
\]
\end{lemma}
\begin{proof}
Since  $\nu < 1$ and $\hat r_t(\tilde x_{t+1}) \leq \nu r(x_t)$ in
\cref{eq:stop},
we obtain from \cref{eq:stop,eq:rhat,eq:defR3}, the triangle inequality, and the nonexpansiveness of $\prox_\Psi$ due to convexity of $\Psi$ in the fourth inequality that
\begin{align*}
(1 - \nu) r_t 
\leq &~ r_t - \hat r_t(\tilde x_{t+1})\\ 
\leq &~ \norm{R_t(x_t) - \hat R_t(\tilde x_{t+1})}\\ 
\leq &~ \norm{x_t - \tilde x_{t+1}} + \norm{\prox_{\Psi}(\tilde x_{t+1} - g_t - H_t (\tilde x_{t+1} - x_t)) - \prox_{\Psi}(x_{t} - g_t)}\\ 
\leq &~ \norm{x_t - \tilde x_{t+1}} + \norm{(\tilde x_{t+1} - g_t - H_t (\tilde x_{t+1} - x_t)) - (x_{t} - g_t)}\\ 
= &~ \norm{x_t - \tilde x_{t+1}} + \norm{(\tilde x_{t+1} - x_{t} - H_t (\tilde x_{t+1} - x_t))}\\ 
\leq &~ (2 + \norm{H_t})\norm{\tilde x_{t+1} - x_t} \\ 
= &~ (2 + \norm{H_t})\norm{\tilde p_t},
\end{align*}
proving the stated result.
\end{proof}

\end{document}